\documentclass{amsart}

\usepackage{paralist}
\usepackage[]{hyperref}
\usepackage{color}
\usepackage{pgf,tikz, tikzscale}
\usetikzlibrary{decorations.pathreplacing,decorations.markings}
\usetikzlibrary{arrows, knots, external, positioning}

\usetikzlibrary{arrows.meta}
\tikzstyle arrowstyle=[scale=2]
\tikzstyle mdirected=[postaction={decorate,decoration={markings,
    mark=at position.5 with {\arrow[arrowstyle]{stealth}}}}]
\tikzstyle edirected=[postaction={decorate,decoration={markings,
    mark=at position 1.0 with {\arrow[arrowstyle]{stealth}}}}]

\tikzset{bullet/.style={
shape = circle,fill = black, inner sep = 0pt, outer sep = 0pt, minimum size = 0.35em, line width = 0pt, draw=black!100}}

\usetikzlibrary{circuits.logic.US,circuits.logic.IEC,fit}

\usepackage{tabularx}
\usepackage{array}
\newcolumntype{C}[1]{>{\centering\arraybackslash}m{#1}}

\newtheorem{thm}{Theorem}[section]
\newtheorem{lem}[thm]{Lemma}

\newtheorem{prop}[thm]{Proposition}

\theoremstyle{definition}

\numberwithin{equation}{section}

\theoremstyle{definition}
\newtheorem{definition}{Definition}[section]
\theoremstyle{definition}
\newtheorem{remark}{Remark}[section]
\theoremstyle{definition}

\newtheorem{example}{Example}[section]
\numberwithin{equation}{section}

\newcommand{\RN}[1]{%
  \textup{\uppercase\expandafter{\romannumeral#1}}%
}

\begin{document}

\title[On symplectic fillings of small Seifert $3$-manifolds]{On symplectic fillings of small Seifert $3$-manifolds}

\author{Hakho Choi}
\address{School of Mathematics, Korea Institute for Advanced Study, 85 Hoegiro, Dongdaemun-gu, Seoul 02455, Republic of Korea }
\email{hakho@kias.re.kr}

\author{Jongil Park}
\address{Department of Mathematical Sciences, Seoul National University, Seoul 08826, Republic of Korea }
\email{jipark@snu.ac.kr}

\thanks{}
\date{May, 2019; revised at November, 2021 and July, 2022}

\begin{abstract}
In this paper, we investigate the minimal symplectic fillings of small Seifert 3-manifolds with a canonical contact structure. As a result, we list all minimal symplectic fillings using curve configurations for small Seifert 3-manifolds satisfying certain conditions. Furthermore, we also demonstrate that every such a minimal symplectic filling is obtained by a sequence of rational blowdowns from the minimal resolution of the corresponding weighted homogeneous complex surface singularity. 
\end{abstract}

\maketitle
\hypersetup{linkcolor=black}

\section{Introduction}

One of the fundamental problems in symplectic 4-manifold topology is to classify symplectic fillings of certain 3-manifolds equipped with a natural contact structure. Among them, researchers have long studied symplectic fillings of the link of a normal complex surface singularity. Note that the link of a normal surface singularity carries a canonical contact structure also known as the Milnor fillable contact structure.  
For example, P.~Lisca~\cite{Lis}, M.~Bhupal and K.~Ono~\cite{BOn}, and the second author et al.~\cite{PPSU} completely classified all minimal symplectic fillings of lens spaces and certain small Seifert 3-manifolds coming from the link of quotient surface singularities. L. Starkston~\cite{Sta1} also investigated minimal symplectic fillings of the link of some weighted homogeneous surface singularities.

On the one hand, topologists working on 4-manifold topology are also interested in finding a surgical interpretation for symplectic fillings of a given 3-manifold. More specifically, one may ask whether there is any surgical description of those fillings. 
In fact, it has been known that \emph{rational blowdown} surgery, introduced by R.~Fintushel and R.~Stern~\cite{FS} and generalized by the second author~\cite{Par} and A.~Stipsicz, Z.~Szab{\'o} and J.~Wahl~\cite{SSW}, is a powerful tool to answer this question. 
For example, for the link of quotient surface singularities equipped with a canonical contact structure, it was proven~\cite{BOz}, \cite{CP} that every minimal symplectic filling is obtained by a sequence of rational blowdowns from the minimal resolution of the singularity. On the other hand, L.~Starkston~\cite{Sta2} showed that there are symplectic fillings of some Seifert 3-manifolds that cannot be obtained by a sequence of rational blowdowns from the minimal resolution of the singularity. Note that Seifert 3-manifolds can be viewed as the link of weighted homogeneous surface singularities. 
Hence, it is an intriguing question as to which Seifert 3-manifolds have a rational blowdown interpretation for their minimal symplectic fillings.

 In this paper, we investigate the minimal symplectic fillings of small Seifert 3-manifolds over the 2-sphere satisfying certain conditions. By a \textit{small} Seifert (fibered) $3$-manifold, we assume that it admits at most $3$ singular fibers when it is considered as an $S^1$-fibration over the 2-sphere. 
 In general, a Seifert 3-manifold as an $S^1$-fibration over a Riemann surface can have any number of singular fibers. We denote a small Seifert $3$-manifold $Y$ by $Y(-b; (\alpha_1, \beta_1), (\alpha_2, \beta_2), (\alpha_3, \beta_3))$ whose surgery diagram is given in Figure~\ref{small} and which is also given as a boundary of a plumbing 4-manifold of disk bundles of 2-spheres according to the graph $\Gamma$ in Figure~\ref{small}.
The integers $b_{ij}\geq 2$ in Figure~\ref{small} are uniquely determined by the following continued fraction:
 $$\frac{\alpha_i}{\beta_i}=[ b_{i1}, b_{i2}, \dots, b_{ir_i} ]=b_{i1}-\displaystyle {\frac{1}{b_{i2}-\displaystyle\frac{1}{\cdots-\displaystyle\frac{1}{b_{ir_i}}}}}$$
\begin{figure}[h]
\begin{tikzpicture}[scale=0.6]
\begin{scope}
\begin{knot}[
	clip width=5,
	clip radius = 2pt,
	end tolerance = 1pt,
]
\strand (0,0) ellipse (3 and 1.5);
\strand (-1.5,-1.75) ellipse (0.25 and 0.8);
\strand (0,-1.75) ellipse (0.25 and 0.8);
\strand (1.5,-1.75) ellipse (0.25 and 0.8);
\draw (-3,1.5) node {$-b$};
\draw (-1.5,-2.55) node[below] {$-\frac{\alpha_1}{\beta_1}$};
\draw (0,-2.55) node[below] {$-\frac{\alpha_2}{\beta_2}$};
\draw (1.5,-2.55) node[below] {$-\frac{\alpha_3}{\beta_3}$};

\flipcrossings{1,4,5}
\end{knot}
\end{scope}
\begin{scope}[shift={(8,0)}]
\node[bullet] at (0,1.5){};
\node[bullet] at (0,0){};
\node[bullet] at (-2,0){};
\node[bullet] at (2,0){};
\node[bullet] at (0,-1){};
\node[bullet] at (-2,-1){};
\node[bullet] at (2,-1){};
\node[bullet] at (0,-3){};
\node[bullet] at (-2,-3){};
\node[bullet] at (2,-3){};
\draw (0,1.5)--(0,-1.5);
\draw (0,1.5)--(2,0)--(2,-1.5);
\draw (0,1.5)--(-2,0)--(-2,-1.5);
\draw[dotted](0,-1.5)--(0,-2.5);
\draw[dotted](2,-1.5)--(2,-2.5);
\draw[dotted](-2,-1.5)--(-2,-2.5);
\draw(0,-2.5)--(0,-3);
\draw(2,-2.5)--(2,-3);
\draw(-2,-2.5)--(-2,-3);
\draw (0,1.5) node[above] {$-b$};
\draw (0,0) node[left] {$-b_{21}$};
\draw (0,-1) node[left] {$-b_{22}$};
\draw (0,-3) node[left] {$-b_{2r_2}$};
\draw (-2,0) node[left] {$-b_{11}$};
\draw (-2,-1) node[left] {$-b_{12}$};
\draw (-2,-3) node[left] {$-b_{1r_1}$};
\draw (2,0) node[left] {$-b_{31}$};
\draw (2,-1) node[left] {$-b_{32}$};
\draw (2,-3) node[left] {$-b_{3r_3}$};
\end{scope}

\end{tikzpicture}
\caption{Surgery diagram of $Y$ and its associated plumbing graph $\Gamma$}
\label{small}
\end{figure}

If the intersection matrix of a plumbing graph $\Gamma$ is negative definite, which is always true for $b\geq 3$, then there is a canonical contact structure on $Y$ induced from a symplectic structure of the plumbing 4-manifold, where each vertex corresponds to a symplectic 2-sphere and each edge represents an orthogonal intersection between the symplectic 2-spheres ~\cite{GS2}. 
Furthermore, the canonical contact structure on $Y$ is contactomorphic to the contact structure defined by the complex tangency of a complex structure on the link of the corresponding singularity, which is called the $\emph{Milnor fillable}$ contact structure ~\cite{PS}. 

This paper aims to classify all possible list of minimal symplectic fillings of small Seifert 3-manifolds satisfying certain conditions and to prove that every such a minimal symplectic filling is obtained by a sequence of rational blowdowns from the minimal resolution of the corresponding weighted homogeneous surface singularity as it is true for a quotient surface singularity. 
Our strategy is as follows:
For a given minimal symplectic filling $W$ of $Y(-b; (\alpha_1, \beta_1), (\alpha_2, \beta_2), (\alpha_3, \beta_3))$ with $b\geq 4$, we glue $W$ with a concave cap $K$ to get a closed symplectic $4$-manifold $X$.  Then, since the concave cap $K$ always contains an embedded $(+1)$  $2$-sphere corresponding the central vertex, $X$ is a rational symplectic $4$-manifold~\cite{McD}. 
 Furthermore, the adjunction formula and intersection data impose a constraint on the homological data of $K$  in $X\cong \mathbb{CP}^2\sharp N\overline{\mathbb{CP}^2}$.
Under blowing-downs along all exceptional $2$-spheres away from the $(+1)$ $2$-sphere in $X\cong \mathbb{CP}^2\sharp N\overline{\mathbb{CP}^2}$, the concave cap $K$ becomes a neighborhood of symplectic $2$-spheres which are isotopic to $b$ number of  complex lines through symplectic $2$-spheres in $\mathbb{CP}^2$  (\cite{Sta1}, \cite{Sta2} for details). 
Since the symplectic deformation type of $W \cong X \setminus K$ is determined by  the isotopy class of a symplectic embedding of $K$  within a fixed homological embedding, we investigate a symplectic embedding of $K$ using  a \emph{curve configuration} corresponding to $W$, which consists of strands representing irreducible components of $K$ and exceptional $2$-spheres intersecting them (Definition~\ref{definition} and Figure~\ref{curveconfig-example}). 
Since the curve configuration corresponding to $W$ determines a symplectic embedding of $K$, we can recover all minimal symplectic fillings by investigating all possible curve configurations of $Y$.
Sometimes, we find a certain chain of symplectic $2$-spheres lying in $W$, which can be rationally blowing down, from the homological data of $K$.
Note that by rationally blowing down the chain of symplectic $2$-spheres lying in $W$, we obtain another minimal symplectic $W'$ from $W$. 
In this case, we keep track of changes in the homological data of $K$ so that we get a curve configuration of $W'$ from that of $W$.
Finally, by analyzing the effect of rational blowdown surgery on the curve configuration of minimal symplectic fillings, we obtain the following main result.

\begin{thm}
For a small Seifert $3$-manifold $Y(-b; (\alpha_1, \beta_1), (\alpha_2, \beta_2),  (\alpha_3, \beta_3))$ with its canonical contact structure
and $b \geq 4$, all minimal symplectic fillings of $Y$ are listed explicitly by curve configurations. Furthermore, they are also obtained by a sequence of rational blowdowns from the minimal resolution of the corresponding weighted homogeneous surface singularity.
\label{mainThm}
\end{thm}

\begin{remark}
L.~Starkston~\cite{Sta1} originally described a general scheme how to obtain minimal symplectic fillings from given homological data of an embedding of the cap and got some results in special cases. 
L.~Starkston~\cite{Sta2} also showed that the isotopy type of a symplectic line arrangement is uniquely (up to deformation equivalent) determined by its intersection data in the cases that multi-intersection points of a symplectic line arrangement satisfy some mild conditions, which contain the cases appeared in Proposition~\ref{prop-3.2}.
Thus, by combining Proposition~\ref{lemcurveconfiguration} and Proposition~\ref{prop-3.2} in Section 3 with Starkston's result (Proposition 4.2 in~\cite{Sta2}), we concluded that there exists at most one minimal symplectic filling for each possible curve configuration. And then, we proved in Section 4 that every such a curve configuration gives the corresponding minimal symplectic filling, which implies the first statement in Theorem~\ref{mainThm} above. 
\end{remark}



\subsection*{Acknowledgements}

The authors would like to thank anonymous referees for their valuable comments. Jongil Park was supported by Samsung Science and Technology Foundation under Project Number SSTF-BA1602-02 and by the National Research Foundation of Korea(NRF) grant funded by the Korea government (No.2020R1A5A1016126). He also holds a joint appointment in the Research Institute of Mathematics, SNU.

%
%
%

\section{Preliminaries}

\subsection{Weighted homogeneous surface singularities and Seifert $3$-manifolds}
We briefly recall some basics of weighted homogeneous surface singularities and Seifert $3$-manifolds (\cite{Orl} for details). Suppose that $(w_0, \dots, w_n)$ are nonzero rational numbers. A polynomial $f(z_0,\dots, z_n)$ is called \emph{weighted homogeneous} of type $(w_0, \dots, w_n)$ if it can be expressed as a linear combination of monomials $z_0^{i_0}\cdots z_n^{i_n}$ for which
$$i_0/w_0+i_1/w_1+\cdots+i_n/w_n=1.$$
Equivalently, there exist nonzero integers $(q_0,\dots, q_n)$ and a positive integer $d$ satisfying $f(t^{q_0}z_0,\dots t^{q_n}z_n)=t^df(z_0,\dots, z_n)$. Then, a weighted homogeneous surface singularity $(X,0)$ is a normal surface singularity that is defined as the zero loci of weighted homogeneous polynomials of the same type. Note that there is a natural $\mathbb{C}^*$-action given by 
$$t\cdot(z_0,\dots, z_n)=(t^{q_0}z_0,\dots t^{q_n}z_n)$$
with a single fixed point $0 \in X$. This $\mathbb{C}^*$-action induces a fixed point free $S^1\subset \mathbb{C}^*$ action on the link $L:=X\cap \partial B$ of the singularity, where $B$ is a small ball centered at the origin. Hence, the link $L$ is a Seifert fibered $3$-manifold over a genus $g$ Riemann surface, denoted by $Y(-b;g; (\alpha_1, \beta_1), (\alpha_2, \beta_2), \dots, (\alpha_k, \beta_k))$ for some integers $b, \alpha_i$ and $\beta_i$ with $0<\beta_i<\alpha_i$ and $(\alpha_i, \beta_i)=1$. 
Note that $k$ is the number of singular fibers, and there is an associated star-shaped plumbing graph $\Gamma$: the central vertex has genus $g$ and weight $-b$, and each vertex in $k$ arms has genus $0$ and weight $-b_{ij}$ uniquely determined by the following continued fraction $$\frac{\alpha_i}{\beta_i}=[ b_{i1}, b_{i2}, \dots, b_{ir_i} ]=b_{i1}-\displaystyle {\frac{1}{b_{i2}-\displaystyle\frac{1}{\cdots-\displaystyle\frac{1}{b_{ir_i}}}}}$$ with $b_{ij}\geq 2$. 
For example, Figure~\ref{small} shows the case of $g=0$ and $k=3$, which is called a \emph{small} Seifert (fibered) $3$-manifold. By P.~Orlik and P.~Wagreich~\cite{OW}, it is well known that the plumbing graph $\Gamma$ is a dual graph of the minimal resolution of $(X, 0)$. Conversely, if the intersection matrix of $\Gamma$ is negative definite, there is a weighted homogeneous surface singularity whose dual graph of the minimal resolution is $\Gamma$ ~\cite{Pin}. 
Note that a Seifert $3$-manifold $Y$, as a boundary of a plumbed $4$-manifold according to $\Gamma$, inherits a canonical contact structure providing that each vertex represents a symplectic 2-sphere, all intersections between them are orthogonal, and the intersection matrix of $\Gamma$ is negative definite ~\cite{GS2}. Furthermore, if the Seifert $3$-manifold $Y$ can be viewed as the link $L$ of a weighted homogeneous surface singularity, then the canonical contact structure above is contactomorphic to the \emph{Milnor fillable} contact structure, which is given by $TL\cap JTL$ ~\cite{PS}. 

\subsection{Rational blowdowns and symplectic fillings}
Rational blowdown surgery, first introduced by R.~Fintushel and R.~Stern~\cite{FS}, is one of the most powerful cut-and-paste techniques which replaces a certain linear plumbing $C_p$ of disk bundles over a $2$-sphere whose boundary is a lens space $L(p^2, p-1)$ with a rational homology $4$-ball $B_p$, which has the same boundary. 
\begin{figure}[h]
\begin{center}
\begin{tikzpicture}[scale=0.95]
\filldraw (0,0) circle (2pt);
\filldraw (1,0) circle (2pt);
\filldraw (3,0) circle (2pt);
\filldraw (4,0) circle (2pt);
\draw (2,0) node {$\cdots$};
\draw (-0.2,0) node[above] {$-(p+2)$};
\draw (1,0) node[above] {$-2$};
\draw (3,0) node[above] {$-2$};
\draw (4,0) node[above] {$-2$};
\draw (0,0)--(1,0);
\draw (3,0)--(4,0);
\draw (1,0)--(1.5,0) (2.5,0)--(3,0);
\end{tikzpicture}
\end{center}
\caption{Linear plumbing $C_p$}
\end{figure}
Later, Fintushel-Stern's rational blowdown surgery was generalized by J.~Park~\cite{Par} using a configuration $C_{p,q}$ obtained by linear plumbing disk bundles over a $2$-sphere according to the dual resolution graph of $L(p^2,pq-1)$, which also bounds a rational homology $4$-ball $B_{p,q}$. 
In the case of a symplectic $4$-manifold $(X, \omega)$, rational blowdown surgery can be performed in the symplectic category: If all 2-spheres in the plumbing graph are symplectically embedded and their intersections are $\omega$-orthogonal, then the surgered $4$-manifold $X_{p,q}=(X-C_{p,q})\cup B_{p,q}$ also admits a symplectic structure induced from the symplectic structure of $X$ \cite{Sym1}, \cite{Sym2}. In fact, the rational homology $4$-ball $B_{p,q}$ admits a symplectic structure compatible with the canonical contact structure on the boundary $L(p^2,pq-1)$. More generally, in addition to the linear plumbing of $2$-spheres, there is a plumbing of $2$-spheres according to star-shaped plumbing graphs with $3$- or $4$-legs admitting a symplectic rational homology $4$-ball~\cite{SSW}, \cite{BS}. That is, the corresponding Seifert $3$-manifold $Y(-b, (\alpha_1,\beta_1), (\alpha_2,\beta_2), (\alpha_3,\beta_3))$, or $Y(-b, (\alpha_1,\beta_1), (\alpha_2,\beta_2), (\alpha_3,\beta_3), (\alpha_4,\beta_4))$ with a canonical contact structure has a minimal symplectic filling whose rational homology is isomorphic to that of the $4$-ball \cite{GS1}. For example, a plumbing graph $\Gamma_{p,q,r}$ in Figure~\ref{Gammapqr} can be rationally blowdown. We will use this later in the proof of the main theorem.
\begin{figure}[h]
\begin{tikzpicture}[scale=0.6]
\node[bullet] at (0,0){};
\node[bullet] at (1,0){};
\node[bullet] at (3,0){};
\node[bullet] at (4,0){};
\node[bullet] at (5,0){};
\node[bullet] at (7,0){};
\node[bullet] at (8,0){};

\node[bullet] at (4,-1){};
\node[bullet] at (4,-3){};
\node[bullet] at (4,-4){};

\node[below left] at (0,0){$-(p+3)$};
\node[above] at (1,0){$-2$};
\node[above] at (3,0){$-2$};
\node[above] at (4,0){$-4$};
\node[above] at (5,0){$-2$};
\node[above] at (7,0){$-2$};
\node[below right] at (8,0){$-(q+3)$};

\node[left] at (4,-1){$-2$};
\node[left] at (4,-3){$-2$};
\node[below] at (4,-4){$-(r+3)$};

\node at (2,0){$\cdots$};
\node at (6,0){$\cdots$};
\node at (4,-1.9){$\vdots$};

\draw (0,0)--(1.5,0);
\draw (2.5,0)--(5.5,0);
\draw (6.5,0)--(8,0);
\draw (4,0)--(4,-1.5);
\draw (4,-2.5)--(4,-4);

	\draw [thick,decorate,decoration={brace,mirror,amplitude=5pt},xshift=0pt,yshift=-7pt]
	(1,0) -- (3,0) node [black,midway,yshift=-11pt] 
	{$q$};

	\draw [thick,decorate,decoration={brace,mirror,amplitude=5pt},xshift=0pt,yshift=-7pt]
	(5,0) -- (7,0) node [black,midway,yshift=-11pt] 
	{$r$};

	\draw [thick,decorate,decoration={brace,amplitude=5pt},xshift=0pt,xshift=7pt]
	(4,-1) -- (4,-3) node [black,midway,xshift=11pt] 
	{$p$};

\end{tikzpicture}
\caption{Plumbing graph $\Gamma_{p,q,r}$}
\label{Gammapqr}
\end{figure}

As rational blowdown surgery does not affect the symplectic structure near the boundary, if there is a plumbing of disk bundles over symplectically embedded $2$-spheres that can be rationally blown down, then one can obtain another symplectic filling by replacing the plumbing with a rational homology $4$-ball. 
In the case of the link of quotient surface singularities, it was proven~\cite{BOz}, \cite{CP} that every minimal symplectic filling is obtained by a sequence of rational blowdowns from the minimal resolution of the singularity, which is diffeomorphic to a plumbing of disk bundles over symplectically embedded $2$-spheres:
First, they constructed a genus-$0$ or genus-$1$ Lefschetz fibration $X$ on each minimal symplectic filling of the link of a quotient surface singularity. Suppose that $w_1$ and $w_2$ are two words consisting of right-handed Dehn twists along curves in a generic fiber that represent the same element in the mapping class group of the generic fiber. If the monodromy factorization of $X$ is given by $w_1\cdot w'$, one can construct another Lefschetz fibration $X'$ whose monodromy factorization is given by $w_2\cdot w'$. The operation of replacing $w_1$ with $w_2$ is called a \emph{monodromy substitution}. 
Next, they showed that the monodromy factorization of each minimal symplectic filling of the link of a quotient surface singularity is obtained by a sequence of monodromy substitutions from that of the minimal resolution. 
Furthermore, these monodromy substitutions can be interpreted as rational blowdown surgeries topologically. Note that all rational blowdown surgeries mentioned here are linear: a certain linear chain $C_{p,q}$ of $2$-spheres is replaced with a rational homology $4$-ball.   

\subsection{Minimal symplectic fillings of small Seifert $3$-manifold}
\label{Sta}
In this subsection, we briefly review Starkston's results~\cite{Sta1}, ~\cite{Sta2} for minimal symplectic fillings of a small Seifert fibered $3$-manifold $Y(-b; (\alpha_1, \beta_1), (\alpha_2, \beta_2), (\alpha_3, \beta_3))$ with $b\geq4$. The condition $b\geq4$ on the weight (equivalently, degree) of a central vertex of the plumbing graph $\Gamma$ ensures that one can always choose a concave cap $K$, which is also star-shaped, with a $(+1)$ central $2$-sphere and $(b-4)$ 
arms, each of which consists of a single $(-1)$ $2$-sphere as in Figure~\ref{cap}. 
Here $[ a_{i1}, a_{i2}, \dots, a_{in_i} ]$ denotes a dual continued fraction of $[ b_{i1}, b_{i2}, \dots, b_{ir_i} ]$, that is,  $\frac{\alpha_i}{\alpha_i-\beta_i}=[ a_{i1}, a_{i2}, \dots, a_{in_i} ]$ while $\frac{\alpha_i}{\beta_i}=[ b_{i1}, b_{i2}, \dots, b_{ir_i} ]$.

\begin{figure}[h]
\begin{tikzpicture}[scale=1.2]
\draw(0,0)--(5.2,0);
\draw (0.5,0.2)--(0,-0.7);
\draw (1.5,0.2)--(1,-0.7);
\draw (2.5,0.2)--(2,-0.7);
\draw (3.9,0.2)--(3.4,-0.7);
\draw (4.9,0.2)--(4.4,-0.7);
\node at (4.,-0.6) {$\cdots$};

	\draw [decorate,decoration={brace,amplitude=5pt,mirror},xshift=2pt,yshift=-3pt]
	(3.4,-0.7) -- (4.4,-0.7) node [black,midway,yshift=-10pt] 
	{\footnotesize $b-4$};

\draw (0,-0.3)--(0.5,-1.2);
\draw (1,-0.3)--(1.5,-1.2);
\draw (2,-0.3)--(2.5,-1.2);

\node at (0.25, -1.5) {$\vdots$};
\node at (1.25, -1.5) {$\vdots$};
\node at (2.25, -1.5) {$\vdots$};

\draw (0,-1.8)--(0.5,-2.7);
\draw (1,-1.8)--(1.5,-2.7);
\draw (2,-1.8)--(2.5,-2.7);

\node[left] at (0,0) {$+1$};
\node[right] at (0.25, -0.25) {$-a_{11}$};
\node[right] at (1.25, -0.25) {$-a_{21}$};
\node[right] at (2.25, -0.25) {$-a_{31}$};
\node[right] at (3.65, -0.25) {$-1$};
\node[right] at (4.65, -0.25) {$-1$};

\node[right] at (0.25, -0.75) {$-a_{12}$};
\node[right] at (1.25, -0.75) {$-a_{22}$};
\node[right] at (2.25, -0.75) {$-a_{32}$};

\node[right] at (0.25, -2.25) {$-a_{1n_1}$};
\node[right] at (1.25, -2.25) {$-a_{2n_2}$};
\node[right] at (2.25, -2.25) {$-a_{3n_3}$};

\end{tikzpicture}
\caption{Concave cap $K$}
\label{cap}
\end{figure}


For a given minimal symplectic filling $W$ of $Y$, we glue $W$ and $K$ along $Y$ to get a closed symplectic $4$-manifold $X$. Then, the existence of a $(+1)$ $2$-sphere implies that $X$ is a rational symplectic $4$-manifold and, after a finite number of blowing-downs, $X$ becomes $\mathbb{CP}^2$ and the $(+1)$ $2$-sphere in $K$ becomes a complex line $\mathbb{CP}^1 \subset \mathbb{CP}^2$ (see Mcduff~\cite{McD} for details). Under these circumstances, it is natural to ask  the following question: What is the image of $K$ in $\mathbb{CP}^2$ under blowing-downs? In the case that $K$ is linear, which means that the corresponding $Y$ is a lens space, Lisca showed that the image of $K$ is two symplectic $2$-spheres in $\mathbb{CP}^2$, each of which is homologous to $\mathbb{CP}^1 \subset \mathbb{CP}^2$. By analyzing the proof of Lisca's result (Theorem 4.2 in \cite{Lis}), Starkston showed that the image of $K$ is $b$ symplectic $2$-spheres in $\mathbb{CP}^2$, each of which is homologous to $\mathbb{CP}^1 \subset \mathbb{CP}^2$~\cite{Sta1}. For the complete classification of minimal symplectic fillings of $Y$, one needs to classify the isotopy classes of these $b$ symplectic $2$-spheres, which called \emph{symplectic line arrangements}. Since all these spheres are $J$-holomorphic for some $J$ tamed by standard K\"{a}hler form of $\mathbb{CP}^2$ and are homologous to $\mathbb{CP}^1 \subset \mathbb{CP}^2$, they intersect each other at a single point for each pair of $2$-spheres. Note that these intersection points need not be all distinct. These intersection data of a symplectic line arrangement are determined by the homological data of $K$, which also have constraints from adjunction formula. In ~\cite{Sta2}, Starkston showed that symplectic line arrangements with certain types of intersections are isotopic to \emph{complex line arrangements}, that is, the corresponding $b$ symplectic $2$-spheres are isotopic (through symplectic spheres) to $b$ complex lines in $\mathbb{CP}^2$. For example, Starkston classified minimal symplectic fillings by an explicit computation of all possible homological embeddings of $K$ for some families of Seifert fibered spaces (Section 3 and 4.4 in ~\cite{Sta1} and Section 5 in ~\cite{Sta2}). 

 

%
%
%

\section{Strategy for main theorem}

As we saw in the previous section, for each minimal symplectic filling $W$ of $Y$, we obtain a rational symplectic $4$-manifold $X$ which is symplectomorphic to $\mathbb{CP}^2\sharp N\overline{\mathbb{CP}^2}$ for some integer $N$ by gluing $K$ to $W$ along $Y$.
Conversely, given an embedding of a concave cap $K$ into $\mathbb{CP}^2\sharp N\overline{\mathbb{CP}^2}$, we obtain a symplectic filling $W$ of $Y$ by taking a complement of $K$ in $\mathbb{CP}^2\sharp N\overline{\mathbb{CP}^2}$. So the classification of minimal symplectic fillings of $Y$ is equivalent to the classification of the embeddings of $K$ into $\mathbb{CP}^2\sharp N\overline{\mathbb{CP}^2}$ for some $N$. Hence, in order to investigate minimal symplectic fillings $W$ of $Y$, we first introduce two notions, \emph{homological data} and \emph{curve configuration} of corresponding embedding of $K$, which are defined as follows:

\begin{definition}
Suppose $W$ is a minimal symplectic filling  of a small Seifert $3$-manifold $Y$ equipped with a concave cap $K$. Then we have an embedding of $K$ into a rational symplectic $4$-manifold $X\cong\mathbb{CP}^2\sharp N\overline{\mathbb{CP}^2}$ such that the $(+1)$ 2-sphere in $K$ is identified with $\mathbb{CP}^1\subset \mathbb{CP}^2$. Let $l$ be a homology class represented by a complex line $\mathbb{CP}^1$ in $\mathbb{CP}^2$ and $e_i$ be homology classes of exceptional spheres coming from blowing-ups. Then $\{l,e_1,\dots, e_N\}$ becomes a basis for $H_2(X;\mathbb{Z})$, so that the homology class of each irreducible component of $K$ can be expressed in terms of this basis, which we call the \emph{homological data of $K$} for $W$. 
\label{definition}
\end{definition}

 Note that $K$ is sympelctically embedded in $X\cong\mathbb{CP}^2\sharp N\overline{\mathbb{CP}^2}$ and each irreducible component of $K$ can be assumed to be $J$-holomorphic for some $J$ tamed by standard K\"{a}hler form on $X$. Then, there is a sequence of blow-downs from $X$ to $\mathbb{CP}^2$ and we can find a $J$-holomorphic exceptional sphere $\Sigma_i$ whose homology class is $e_i$ disjoint from the central $(+1)$ 2-sphere of $K$ at each stage of blow-downs.  Because of the $J$-holomorphic condition and homological restrictions from the adjunction formula together with intersection data of $K$, the exceptional sphere $\Sigma_i$ intersects positively at most once with the image of an irreducible component of $K$ or is one of the image of irreducible components of $K$ (Proposition 4.4 in ~\cite{Lis}). In particular, for each image $C_j$ of irreducible components of $K$, the intersection number between $e_i$ and $[C_j]$ lies in $\{-1,0,1\}$. 
As mentioned in the previous section, we finally get a symplectic line arrangement in $\mathbb{CP}^2$ which consists of $J$-holomophic 2-spheres, each of which is the image of the first component of each arm under blow-downs. 
The intersection data of the symplectic line arrangement are determined by the homological data of $K$, so that it can be represented as a configuration of strands: Each of strands represents a $J$-holomorphic 2-sphere of a symplectic line arrangement in $\mathbb{CP}^2$ while the intersection of two strands represents a geometric intersection of two 2-spheres. Then, starting from the configuration of the symplectic line arrangement, we can draw a configuration $C$ of strands with degrees by blowing-ups according to the homological data of $K$ until we get $K$ in the configuration. Here the degree of each strand in $C$ means a self-intersection number of the strand. 
To be more precise, when we blow up a point $p$ on a strand in a configuration, we introduce a new strand with degree $-1$ to the point $p$ so that we resolve intersection of strands at $p$ and we decrease the degree of the strands containing $p$ by one.
Hence the configuration $C$, which represents the total transform of a symplectic line arrangement, contains strands representing irreducible components of $K$ and exceptional $(-1)$ $2$-spheres intersecting with the irreducible components.  
We say that two configurations $C_1$ and $C_2$ for $W$ are \emph{equivalent} if there is a bijective map between $(-1)$ strands preserving intersections with the irreducible components of $K$.

\begin{definition}
If there are no strands with degree less than or equal to $-2$ in $C$ except for irreducible components of $K$, we call the configuration $C$ the \emph{curve configuration} of a minimal symplectic filling $W$.
\end{definition}

\begin{remark}
Note that a curve configuration $C$ of $W$ consists of strands representing irreducible components of $K$ and exceptional $2$-spheres intersecting with the irreducible components of $K$. We denote the exceptional $2$-spheres by 
dash-dotted strands (See Figure~\ref{curveconfig-example} for example). 
\begin{figure}[htbp]
\begin{tikzpicture}[scale=0.85]
\begin{scope}
\node[bullet] at (2,0){};
\node[bullet] at (3,0){};
\node[bullet] at (4,0){};
\node[bullet] at (5,0){};
\node[bullet] at (6,0){};

\node[bullet] at (4,-1){};

\node[above] at (2,0){$-3$};
\node[above] at (3,0){$-2$};
\node[above] at (4,0){$-5$};
\node[above] at (5,0){$-4$};
\node[above] at (6,0){$-2$};

\node[left] at (4,-1){$-2$};

\draw (2,0)--(6,0);
\draw (4,0)--(4,-1);
\end{scope}
\begin{scope}[shift={(8,0.5)}]
\draw(0,0.05)--(4,0.05);
\draw[red,thick, dashdotted] (0,-0.09)--(4,-0.09);

\draw (0.5,0.2)--(0,-1.);
\node[right] at (0.2,-1) {$-3$};
\draw[red,thick, dashdotted] (0.1,-0.25)--(0.5,-0.25);
\draw[red,thick, dashdotted] (0.05,-0.35)--(0.45,-0.35);

\draw (1.5,0.2)--(1,-1.);
\node[right] at (1.1,-1) {$-2$};
\draw[red,thick, dashdotted] (1.1,-0.25)--(1.5,-0.25);
\draw[red,thick, dashdotted] (1.05,-0.35)--(1.45,-0.35);

\draw (2.5,0.2)--(2,-1.);
\node[right] at (2.1,-1){$-2$};
\draw[red,thick, dashdotted] (2.1,-0.3)--(2.5,-0.3);

\draw (3.5,0.2)--(3,-1.);
\node[right] at (3.1,-1){$-1$};
\draw[red,thick, dashdotted] (3.1,-0.3)--(3.5,-0.3);

\draw (0,-0.6)--(0.5,-1.8);
\node[right] at (0.5,-1.8) {$-2$};
\draw[red,thick, dashdotted] (0.1,-1.3)--(0.5,-1.3);

\draw (2,-0.6)--(2.5,-1.8);
\node[right] at (2.5,-1.8) {$-2$};
\draw (2.5, -1.4)--(2,-2.6);
\node[right] at (2,-2.6) {$-3$};
\draw[red,thick, dashdotted] (2.05,-2)--(2.45,-2);
\draw[red,thick, dashdotted] (2.0,-2.1)--(2.4,-2.1);

\node[left] at (0,0.05) {$+1$};



\end{scope}
\end{tikzpicture}

\caption{Plumbing graph $\Gamma$ and curve configuration for corresponding concave cap $K$}
\label{curveconfig-example}
\end{figure}

\end{remark}

\begin{remark}
We often use a terminology \emph{configuration of strands} when we deal with an intermediate configuration between a symplectic line arrangement and a curve configuration, or a configuration containing $K$ but there are strands with degree less than or equal to $-2$ other than irreducible components of $K$.
\end{remark}

\begin{prop}
For a given homological data of $K$ for $W$, there is a unique curve configuration $C$ up to equivalence
\label{lemcurveconfiguration}
\end{prop}

\begin{proof}
Since each strand in a curve configuration $C$ represents a $J$-holomorphic 2-sphere for some $J$ tamed by standard K\"{a}hler form on $X\cong\mathbb{CP}^2\sharp N\overline{\mathbb{CP}^2}$, all intersections between the strands represent positive geometric intersections between the corresponding $J$-holomorphic 2-spheres. Note that there is at most one intersection point between any two strands due to homological restrictions. Furthermore, if $e_i$ is a homology class of an exceptional $2$-sphere satisfying $e_i\cdot [C_j]\in\{0,1\}$ for any irreducible component $C_j$ of $K$, then there is a $(-1)$ strand $L_i$ in $C$ whose homology class is $e_i$: Otherwise, there is a blowing-up on the strand $L_i$ so that proper transform of $L_i$ becomes an irreducible component $C_j$ of $K$ whose intersection with $e_i$ is $-1$ contradicting the assumption. Hence there is a $(-1)$ strand $L_i$ representing a $J$-holomorphic exceptional sphere $\Sigma_i$ whose homology class is $e_i$ in $C$ if and only if $e_i\cdot [C_j]\in\{0,1\}$ for any irreducible component $C_j$ of $K$.


Let $C$ and $C'$ be two curve configurations for a fixed homological data of $K$ for $W$. Then, the numbers of $(-1)$ strands in $C$ and $C'$ are equal to the number of $e_i$'s satisfying the condition $e_i\cdot [C_j]\in \{0,1\}$ for any irreducible component $C_j$ of $K$. Hence we can construct a desired bijection between the $(-1)$ strands by finding correspondence between such $e_i$'s and $(-1)$ strands in two curve configurations.


\end{proof}

Now, we investigate minimal symplectic fillings of a given small Seifert $3$-manifold $Y$ by analyzing all the possible curve configurations. For this, we first determine all possible symplectic line arrangements.

\begin{prop}
\label{prop-3.2}
For minimal symplectic fillings of a small Seifert fibered $3$-manifold $Y(-b; (\alpha_1, \beta_1), (\alpha_2, \beta_2), (\alpha_3, \beta_3))$ with $b\geq4$, there are only two possible intersection relations of symplectic line arrangements which can be drawn as in Figure~\ref{linearrange}
\end{prop}

\begin{figure}[h]
\begin{tikzpicture}[scale=1.2]
\begin{scope}
\draw (0.1,0.2)--(2.6,0.2);
\draw (1,0.5)--(1,-1.5);
\draw (1.5,0.5)--(0.5,-1.5);
\draw (0.5,0.5)--(1.5,-1.5);
\draw (2,0.5)--(0,-1.5);
\draw (2.5,0.5)--(-.5,-1.5);
\node at (2.15,0.4) {$\cdots$};
\end{scope}
\begin{scope}[shift={(4.5,0)}]
\draw (-1,0.2)--(2,0.2);
\draw (1,0.5)--(1,-1.5);
\draw (1.5,0.5)--(0.5,-1.5);
\draw (0.5,0.5)--(1.5,-1.5);
\draw (-1,0.5)--(2,-1.5);
\node at (1.25,0.4) {$\cdots$};

\end{scope}

\end{tikzpicture}
\caption{Symplectic line arrangements}
\label{linearrange}
\end{figure}
\begin{proof}
Since $Y$ is a small Seifert $3$-manifolds with $b\geq4$, we can always choose a concave cap $K$ with a $(+1)$ central $2$-sphere and $(b-4)$ 
arms, each of which consists of a single $(-1)$ $2$-sphere as in Figure~\ref{cap}. Furthermore, since the blowing-downs are disjoint from the central $2$-sphere in $K$, each of $(b-1)$ number of arms in $K$ descends to a single $(+1)$ $J$-holomorphic $2$-sphere intersecting at a distinct  point with an image of the central $2$-sphere of $K$ under the blowing-downs. 
Let $C_1, C_2, \dots, C_{b-4}$ be the images of $(b-4)$ number of $(-1)$ $2$-spheres in $K$ under the blowing-downs.
Then, they should have a common intersection point in $\mathbb{CP}^2$: Otherwise, we have distinct two points $p$ and $q$ on some $C_i$ so that $C_i$ intersects $C_j$ and $C_k$ at $p$ and $q$ respectively. Let $r$ be an intersection point of $C_j$ and $C_k$. Then, any $J$-holomorphic $2$-sphere coming from an arm of $K$ other than $C_1,\dots C_{b-4}$ must pass two of $p, q$ and $r$, which is a contradiction. 
\begin{figure}[h]
\begin{tikzpicture}[scale=1.2]
\filldraw (5/12,-0.) circle (1pt);
\filldraw (19/12,-0.) circle (1pt);
\filldraw (1,-7/8.) circle (1pt);

\node at (0.1,0)[left]{$C_i$};
\node at (1.25,-1.25)[right]{$C_j$};
\node at (.75,-1.25)[left]{$C_k$};

\node at (7/12,0)[above]{$p$};
\node at (21/12,0)[below]{$q$};
\node at (.9,-7/8)[left]{$r$};

\draw (0.1,0)--(1.9,0);
\draw (.25,0.25)--(1.25,-1.25);
\draw (1.75,0.25)--(.75,-1.25);
\end{tikzpicture}
\caption{Configuration for $C_i$, $C_j$ and $C_k$}
\end{figure}

If $b\geq 6$, similar argument shows that there is at most one $J$-holomorphic $2$-sphere coming from an arm of $K$ intersecting at a different point from the common intersection point $p$ with $C_i$, which proves the proposition. 

In the case of $b\leq 5$, we can easily check that Figure~\ref{linearrange} gives all possible symplectic line arrangements: If $b=5$, then there is only one $C_1$ coming from $(-1)$ $2$-sphere from $K$. Recall that there are at most two intersection points on $C_1$. If there is only one intersection point on $C_1$, then we get the left-hand figure in Figure~\ref{linearrange}. If there are two intersection points $p$ and $q$ on $C_1$, then two of three $J$-holomorphic $2$-spheres coming from the arms of $K$ other than $C_1$ pass $p$ and the other passes $q$ (or vice versa), so that we get the right-hand figure in Figure~\ref{linearrange}. For $b=4$ case, we have only three strands in a figure for a symplectic line arrangement except the strand from $(+1)$ $2$-sphere so that we have only two possibilities.
\end{proof}

Next, for the complete classification of minimal symplectic fillings of $Y$, we need to consider the isotopy classes of embeddings of $K$ with a fixed homological data in $X\cong \mathbb{CP}^2\sharp N\overline{\mathbb{CP}^2}$. By blowing down $J$-holomorphic $2$-spheres, it descends to isotopic types of corresponding symplectic line arrangement in $\mathbb{CP}^2$. 
By Proposition 4.1 and 4.2 in ~\cite{Sta2}, two symplectic line arrangements in Figure~\ref{linearrange} are actually isotopic to  complex line arrangements through symplectic configurations, which means that there is a unique minimal symplectic filling up to symplectic deformation equivalent for each possible homological data of $K$. 
Since a homological data of $K$ gives a unique curve configuration $C$ up to equivalence by Proposition~\ref{lemcurveconfiguration}, we analyze minimal symplectic fillings of small Seifert $3$-manifold $Y$ by considering all possible curve configurations obtained from the complex line arrangements in Figure~\ref{linearrange}.\\

As previously mentioned, in the case of quotient surface singularities that include all lens spaces and some small Seifert $3$-manifolds, every minimal symplectic filling is obtained by linear rational blowdown surgeries from the minimal resolution of the corresponding singularity. 
However, this is not true anymore for small Seifert $3$-manifolds in general.  For example, a rational homology $4$-ball of $\Gamma_{p,q,r}$ in Figure~\ref{Gammapqr} might not be obtained by linear rational blowdown surgeries. 
Nevertheless, many cases such as $b \geq 5$ are in fact obtained by linear rational blowdowns from their minimal resolutions. For the case of $b=4$, one might need $3$-legged rational blowdown surgeries to get a minimal symplectic filling.
Hence, it is natural to prove the two cases $b \geq 5$ and $b =4$ separately. 

\subsection{$b\geq 5$ case}
We consider all possible curve configurations coming from two complex line arrangements in Figure~\ref{linearrange} which can be divided into three types. First, we need to blow up all intersection points in the line arrangements so that we get two configurations as in Figure~\ref{blowingup1}.
\begin{figure}[h]
\begin{center}
\begin{tikzpicture}[scale=1.4]
\begin{scope}

\draw (0,0.2)--(2,0.2);
\draw (0.4,0.5)--(0.4,-1.5);
\draw (0.8,0.5)--(0.8,-1.5);
\draw (1.2,0.5)--(1.2,-1.5);
\draw (1.6,0.5)--(1.6,-1.5);
\draw[red,thick, dashdotted] (0,-0.6)--(2,-0.6);
\node[left,red,thick, dashdotted] at (0,-0.6) {$-1$};
\node[left] at (0,0.2) {$+1$};
\node at (1,-2) {$(a)$};

\node at (1.43,0.35) {$\cdots$};
\node[above] at (0.4,0.5) {$0$};
\node[above] at (0.8,0.5) {$0$};
\node[above] at (1.2,0.5) {$0$};
\node[above] at (1.6,0.5) {$0$};
\end{scope}

\begin{scope}[shift={(4,0)}]
\node at (1,-2) {$(b)$};

\draw (0,0.2)--(2,0.2);
\draw (0.4,0.5)--(0.4,-1.5);
\draw (0.8,0.5)--(0.8,-0.6);
\draw (1.2,0.5)--(1.2,-0.8);
\draw (1.6,0.5)--(1.6,-1);
\draw[red,thick, dashdotted] (0.6,0)--(1.8,0);
\draw[red,thick, dashdotted](0.3,-0.6)--(0.9,-0.3);
\draw[red,thick, dashdotted](0.3,-1)--(1.3,-0.5);
\draw[red,thick, dashdotted](0.3,-1.4)--(1.7,-0.7);
\node at (1.43,0.35) {$\cdots$};
\node[above] at (0.8,0.5) {$-1$};
\node[above] at (1.2,0.5) {$-1$};
\node[above] at (1.6,0.5) {$-1$};
\node[left] at (0,0.2) {$+1$};
\node [above] at (0.8,-1.15) {$\ddots$}; 
\end{scope}
\end{tikzpicture}
\end{center}
\caption{Blowing-ups of the line arrangements}
\label{blowingup1}
\end{figure}
There are two possibilities for a strand representing exceptional sphere in intermediate configurations coming from blowing-ups : Blow up some intersection points or not. 
Once we blow up an intersection point on a strand representing an exceptional sphere $\Sigma$, which means the proper transform of $\Sigma$ becomes an irreducible component of $K$, we should blow up all the intersection points except one intersection point because each strand intersecting the strand for $\Sigma$ become irreducible components of distinct arms in $K$. We can also blow up the last intersection point we did not blow up to get another curve configuration, but it is not necessary in general.

In case of we blow up an intersection point on the dash-dotted strand of $(a)$ in Figure~\ref{blowingup1}, we get a configuration $(a)'$ in Figure~\ref{startingposition}.
\begin{figure}[h]
\begin{center}
\begin{tikzpicture}[scale=1.3]
\begin{scope}
\node at (1,-2) {$(a)'$};

\draw (0,0.2)--(2,0.2);
\draw (0.4,0.5)--(0,-0.5);
\draw (0,-0.05)--(0.4,-1.6);
\node[right] at (0.4,-1.6) {$c$};
\draw (0.8,0.5)--(0.8,-0.6);
\draw (1.2,0.5)--(1.2,-0.8);
\draw (1.6,0.5)--(1.6,-1);

\draw[red,thick, dashdotted](0,-0.7)--(0.9,-0.25);
\draw[red,thick, dashdotted](0.1,-1.1)--(1.3,-0.5);
\draw[red,thick, dashdotted](0.2,-1.5)--(1.7,-0.75);

\node[left, red] at (0,-0.7) {$e_1$};
\node[left, red] at (0.1,-1.1) {$e_2$};

\node at (1.43,0.35) {$\cdots$};
\node[above] at (0.8,0.5) {$-1$};
\node[above] at (1.2,0.5) {$-1$};
\node[above] at (1.6,0.5) {$-1$};
\node[left] at (0,0.2) {$+1$};
\node[above left] at (0.95,-1.225) {$\ddots$};
\end{scope}

\begin{scope}[shift={(4,0)}]
\node at (1,-2) {$(b)$};

\draw (0,0.2)--(2,0.2);
\draw (0.4,0.5)--(0.4,-1.6);
\node[right] at (0.4,-1.6) {$c$};

\draw (0.8,0.5)--(0.8,-0.6);
\draw (1.2,0.5)--(1.2,-0.8);
\draw (1.6,0.5)--(1.6,-1);
\draw[red,thick, dashdotted] (0.6,0)--(1.8,0);
\draw[red,thick, dashdotted](0.3,-0.6)--(0.9,-0.3);
\draw[red,thick, dashdotted](0.3,-1)--(1.3,-0.5);
\draw[red,thick, dashdotted](0.3,-1.4)--(1.7,-0.7);
\node[left, red] at (0.3,-0.6) {$e_1$};
\node[left, red] at (0.3,-1) {$e_2$};

\node at (1.43,0.35) {$\cdots$};
\node[above] at (0.8,0.5) {$-1$};
\node[above] at (1.2,0.5) {$-1$};
\node[above] at (1.6,0.5) {$-1$};
\node[left] at (0,0.2) {$+1$};
\node [above] at (0.8,-1.15) {$\ddots$}; 
\end{scope}
\end{tikzpicture}
\end{center}
\caption{Two configurations}
\label{startingposition}
\end{figure} 
When we start with two configurations in Figure~\ref{startingposition}, we can assume without loss of generality that the first three arms become \emph{essential arms} in $K$, which consist of strands with degree less than or equal to $-2$. Since the degree of other arms already $-1$, we can only blow up $e_1$ and $e_2$ among dotted exceptional strands. In conclusion, we can divide all the possible curve configurations into following three types.
\begin{itemize}
\item{Type A:}
Curve configurations obtained from $(a)$ in Figure~\ref{blowingup1} without blowing up the  exceptional strand.
\item{Type B:}
Curve configurations obtained from $(a)'$ or $(b)$ in Figure~\ref{startingposition} by blowing up at most one $e_i$ $(1\leq i \leq 2)$.
\item{Type C:}
Curve configurations obtained from $(a)'$ or $(b)$ in Figure~\ref{startingposition} by blowing up both $e_1$ and $e_2$.
\end{itemize}

%

\subsection{$b=4$ case}
We divide all curve configurations for $b=4$ into the following two cases:
\begin{itemize}
\item Curve configurations  of type A, B or C as in $b\geq5$ case.
\item{Type D:} Curve configurations  obtained from $(b)$ in Figure~\ref{startingposition} by blowing up all exceptional $(-1)$ strands.
\end{itemize}
Then, since we can deal with the first case using the same argument in the $b\geq 5$ case,
it suffices to prove Type D case whose corresponding curve configurations are coming from some configurations $C_{p,q,r}$ in Figure~\ref{b4cases}, which are obtained from the right-hand figure in Figure~\ref{startingposition} (See Subsection 4.4 for details). The main difference between  $b=4$ case and $b\geq5$ case is that one can use all three exceptional $2$-spheres to get a concave cap $K$ for $b=4$, while one can use only $e_1$ and $e_2$ for $b\geq5$ from the right-hand figure in Figure~\ref{startingposition}.

\begin{figure}[h]

\begin{tikzpicture}[scale=.9]
\begin{scope}
\draw (-0.5,0)--(3.5,0);
\node[left] at (-0.5,0){$+1$};
\draw (0,0.5)--(0,-3);
\node[above] at (0,0.5){$-2$};
\draw (1.5,0.5)--(1.5,-2);
\node[above] at (1.5,0.5){$-2$};
\draw (3,0.5)--(3,-3);
\node[above] at (3,.5){$-2$};

\draw (-0.2,-1.2)--(0.95,-2.35);
\node[below] at (0.95,-2.35){$-2$};

\draw[red,thick, dashdotted] (1.7,-1.2)--(0.55,-2.35);
\node[red, above] at (1.025,-1.775){$e_1$};

\draw (1.3,-0.3)--(2.45,-1.45);
\node[below] at (2.45,-1.45){$-2$};
\draw[red,thick, dashdotted] (3.2,-0.3)--(2.05,-1.45);
\node[red, above] at (2.525,-0.875){$e_2$};

\draw[red,thick, dashdotted] (-0.2,-2.1)--(1.7,-4);
\node[red, below] at (0.64,-3.05){$e_3$};
\draw (3.2,-2)--(1.3,-4);
\node[below] at (2.35,-3){$-2$};

\end{scope}

\begin{scope}[shift={(6,0)}]

\draw (-0.5,0)--(3.5,0);
\node[left] at (-0.5,0){$+1$};
\draw (0,0.5)--(0,-3);
\node[above] at (0,0.5){$-(r+2)$};
\draw (1.5,0.5)--(1.5,-2);
\node[above] at (1.5,0.5){$-(p+2)$};
\draw (3,0.5)--(3,-3);
\node[above] at (3,.5){$-(q+2)$};

\draw (-0.1,-1.4)--(0.3,-1.4);
\draw (0.2,-1.3)--(0.2,-1.7);
\node at (0.45,-1.65) {$\ddots$};

\draw (0.45,-1.95)--(0.85,-1.95);
\draw (0.75,-1.85)--(0.75,-2.25);

	\draw [decorate,decoration={brace,amplitude=5pt},yshift=10pt,xshift=5pt]
	(0,-1.5) -- (0.75,-2.25) node [black,midway,xshift=13pt, yshift=10pt] 
	{$p+1$};

\draw[red,thick, dashdotted] (1.7,-1.2)--(0.55,-2.35);
\node[red, below right] at (1.025,-1.775){$e_1$};


\draw (1.4,-0.5)--(1.8,-0.5);
\draw (1.7,-0.4)--(1.7,-0.8);
\node at (1.95,-0.75) {$\ddots$};

\draw (1.95,-1.05)--(2.35,-1.05);
\draw (2.25,-0.95)--(2.25,-1.35);

	\draw [decorate,decoration={brace,amplitude=5pt},yshift=10pt,xshift=5pt]
	(1.5,-0.6) -- (2.25,-1.35) node [black,midway,xshift=13pt, yshift=9pt] 
	{$q+1$};

\draw[red,thick, dashdotted] (3.2,-0.3)--(2.05,-1.45);
\node[red, below right] at (2.525,-0.875){$e_2$};

\draw[red,thick, dashdotted] (-0.2,-2.1)--(1.7,-4);
\draw(1.55,-3.35)--(1.95,-3.35);
\draw(1.65,-3.65)--(1.65,-3.25);
\draw(1.35,-3.55)--(1.75,-3.55);
\draw(1.45,-3.45)--(1.45,-3.85);

\node[red, below] at (0.64,-3.05){$e_3$};
\draw(3.1,-2.2)--(2.7,-2.2);
\draw(2.8,-2.1)--(2.8,-2.5);
\draw(2.9,-2.4)--(2.5,-2.4);
\draw(2.6,-2.3)--(2.6,-2.7);

	\draw [decorate,decoration={brace,mirror,amplitude=5pt},yshift=-5pt,xshift=5pt]
	(1.65,-3.65) -- (2.9,-2.4) node [black,midway,xshift=10pt, yshift=-10pt] 
	{$r+1$};

\filldraw (2.25,-2.95) circle(0.5pt);
\filldraw (2.4,-2.8) circle(0.5pt);
\filldraw (2.1,-3.1) circle(0.5pt);

\end{scope}
\end{tikzpicture}
\caption{Curve configurations $C_{0,0,0}$ and $C_{p,q,r}$}
\label{b4cases}

\end{figure}

\section{Proof of main theorem}
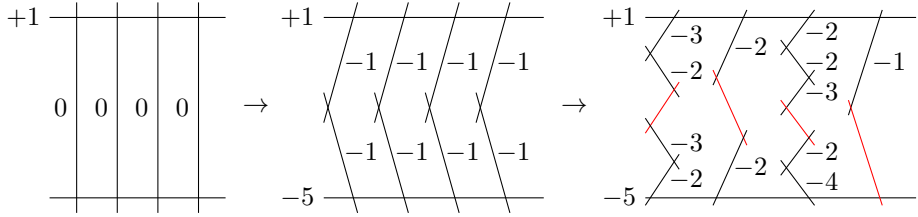
In this section, for a given possible curve configuration $C$, we show that there is a sequence of rational blowdowns from the minimal resolution $\widetilde{M}$ to minimal symplectic filling $W$ of $Y$ corresponding to $C$. Since any minimal symplectic filling of a lens space is obtained by a sequence of rational blowdowns from a linear plumbing which is the minimal resolution corresponding to the lens space \cite{BOz}, it suffices to construct a sequence of curve configurations $C=C_0, C_1, \dots, C_n$ such that each minimal symplectic filling $W_{i}$ corresponding to $C_i$ is obtained from $W_{i+1}$ by replacing a certain linear plumbing $L_i$ with its minimal symplectic filling. Here $C_n$ denotes a curve configuration for the minimal resolution $\widetilde{M}$. 
As previously mentioned, since our possible symplectic line arrangements are isotopic to complex line arrangements, it suffices to work in complex category with a symplectic form $\omega$ coming from the standard K\"{a}hler form on $\mathbb{CP}^2$. 
In order to show that there is a symplectic embedding of $L_i$ in $W_{i+1}$, we construct a configuration $C_{i+1}'$ of strands, which is not a curve configuration for $W_{i+1}$, from a complex line arrangement by blowing-ups with same homological data of $K$ for $W_{i+1}$ so that we have $L_i$ disjoint from $K$ in $C'_{i+1}$. Since we work in complex category, each strand in $C_{i+1}'$ can be considered as a complex rational curve in a rational surface $X$ while the intersections between strands represent positive geometric intersections between the corresponding rational curves. This observation implies that $L_i$ is symplectically embedded in $W_{i+1}$.

 Now we introduce the notion of \emph{standard blowing-ups} which is frequently appeared in the construction of $W_i$ from $W_{i+1}$. Let $K$ and $K'$ be two star-shaped plumbing graphs having the same number of arms together with $(+1)$ central vertex, and let $-a_{ij}$ $(1\leq j \leq n_i)$ and $-a'_{ij}$ $(1\leq j \leq n'_i)$ be the weights (equivalently, degrees) of $j^{\text{th}}$-vertex in the $i^{\text{th}}$-arm of $K$ and $K'$ respectively. We say $K'\leq K$ if $n'_i\leq n_i$ and $a'_{ij}\leq a_{ij}$ for any $i$ and $j$ except for $a'_{in'_i} < a_{in'_i}$ in the case of $n'_i<n_i$. The condition $K'\leq K$ guarantees that we can obtain a configuration of strands representing $K$ by blowing ups from a configuration representing $K'$ in the following way: We blow up non-intersection points of the last component of each $i^{\text{th}}$-arm in $K'$ consecutively until we get $n_i$ components, and then we blow up each component at non-intersection points to get the right weights.
 
\begin{definition}
Let $C'$ be a configuration of strands obtained from a complex line arrangement by blowing-ups containing a star-shaped plumbing graph $K'$ with a homological data.
If $K'\leq K$ and the degree of all strands in $C'\setminus K'$ is $-1$, then we can obtain a curve configuration $\widetilde{C'}$ from $C'$ by blowing-up at non-intersection points only. In this case, we say that the curve configuration $\widetilde{C'}$ is obtained by \emph{standard blowing-ups} from $C'$.
\label{defstandardblowingup}
\end{definition}

\begin{remark}
Note that, with a homological data of $K'$ in $C'$, the standard blowing-ups induce a unique homological data of $K$ for $\widetilde{C'}$:
Let $e$ be a homology class of an exceptional sphere coming from blowing-ups from $C'$ to $\widetilde{C'}$.
Since we blow-up non-intersection points, $e$ appears in at most two $[C_{j_1}^{i_1}]$ and $[C_{j_2}^{i_2}]$ where $C_{j}^i$ denotes $j^{\text{th}}$-component in $i^{\text{th}}$-arm of $K$. Furthermore, if $e$ appears in two $[C_{j_1}^{i_1}]$ and $[C_{j_2}^{i_2}]$, then $i_1=i_2=i$ and $j_2=j_1+1$ with $e\cdot [C_{j_1}^{i}]=1$ and
$e\cdot [C_{j_1+1}^{i}]=-1$. 
\end{remark}

For a given star-shaped plumbing graph $K'\leq K$, in general if $ n'_i < n_i$  for some $i$ where $n_i'$ and $n_i$ are the number of components in $i^{\text{th}}$-arm of $K'$ and $K$ respectively, there are possibly other ways of blowing-ups to get $i^{\text{th}}$-arm of $K$ from that of $K'$. 
Let $C'$ be a configuration of strands containing $K'\leq K$ as in Definition~\ref{defstandardblowingup}. Assume furthermore that $ n'_i < n_i$ for some $i$.  Let $\widetilde{C'}$ be a curve configuration obtained from $C'$ by standard blowing-ups. Then we get the following three fundamental lemmas.


\begin{lem}
Let $C$ be a curve configuration for $K$, and let $W$ be the minimal symplectic filling of $Y$ corresponding
to $C$. Suppose $C'$ is a configuration for $K' \leq K$  such that the standard blowing-ups
$\widetilde{C'}$ of $C'$ differs from $C$ only in the components $C^i_{j}$ for $n'_{i} \leq
j \leq n_i$. Let $\widetilde{W}$ denote the minimal symplectic filling of $Y$
corresponding to $\widetilde{C'}$. Then there is a symplectically embedded linear plumbing $L$ of $2$-spheres
determined by $[b_1, b_2, \dots, b_r]$  in  $\widetilde{W}$ such that $W$ is obtained by  $\widetilde{W}$ by replacing the plumbing $L$ with some minimal filling $W_L$ of the lens space boundary of the linear plumbing $L$. Furthermore,
$[b_1, b_2, \dots, b_r]$ is the dual of $[(a_{in'_i}-a'_{in'_i}), a_{in'_{i}+1}, a_{in'_{i}+2}, \dots a_{in_i}]$, where $-a_{ij}$ and $-a'_{ij}$ are the weights of $j^{\text{th}}$-component in the $i^{\text{th}}$-arm of $K$ and $K'$ respectively.
\label{fundamentallem}
\end{lem}

\begin{proof}
We can 
assume that $a_{in'_i}-a'_{in'_i}\geq 2$ because the way of blowing-ups from $i^{\text{th}}$-arm of $K'$ to that of $K$ remains same when we replace $K'$ with $K''$ where $K''$ is obtained from $K'$ by blow up the last component of the $i^{\text{th}}$-arm. 

First we show that there is a symplectic linear embedding $L$ in $\widetilde{W}$. Let $S$ be a configuration of strands containing $K$ obtained as follows: We blow up the last component in the $i^{\text{th}}$-arm of $K'$ in $C'$ at a non-intersection point so that we have two consecutive strands of degree $-a'_{in'_i}-1$ and $-1$. Since the continued fraction $[b_1, b_2, \dots, b_r]$ is dual to $[(a_{in'_i}-a'_{in'_i}), a_{in'_{i}+1}, a_{in'_{i}+2}, \dots a_{in_i}]$ by the definition of $L$, we obtain a linear chain of strands containing the rest of $i^{\text{th}}$-arm in $K$ and $L$ from the two strands by blowing up consecutively at intersection points as in Figure~\ref{constructL}, so that there is an embedding $L$ in the complement of $K$ in a rational surface $X$. 
Furthermore, since we started from the same homological data of $K'$ in $C'$ and since a blowing-up for $C'$ to $S$ either increases the number of components or decreases the degree of an irreducible component of $K$, 
the homological data of $K$ for both $\widetilde{C'}$ and $S$ are the same, so that there is a symplectic embedding $L$ in $\widetilde{W}$.
\begin{figure}[h]
\begin{tikzpicture}[xscale=0.9]
\begin{scope}

\draw (0,-1.4)--(0.5,0.2);
\draw (0,-1.)--(0.5,-2.6);
\node at (1.2,-0.6){$-a'_{in'_i}-1$};
\node at (0.55,-1.8){$-1$};

\filldraw[red] (1/16,-1.2) circle (1.2pt);


\draw[-latex,line width=1.5pt] (2.9-1.,-1.2)--(4.15-1.,-1.2) ;

\end{scope}

\begin{scope}[shift={(4.25+1.25-1,1.1)},scale=.8]

\node[right] at (0.25, -0.25) {$-a_{in'_i}$};
\node[right] at (0.25, -0.9) {$-a_{in'_i+1}$};
\node[right] at (0.25, -2) {$-a_{in_i}$};
x
\node[right] at (0.25, -3.35) {$-b_{r}$};
\node[right] at (0.25, -4) {$-b_{r-1}$};
\node[right] at (0.25, -5.2) {$-b_{1}$};

\draw (0.5,0.2)--(0,-0.7);

\draw (0,-0.3)--(0.5,-1.2);

\node at (0.25, -1.2) {$\vdots$};

\draw (0.5,-1.6)--(0.,-2.5);
\draw[red,thick, dashdotted] (0.,-2.25)--(0.5,-3.15);
\draw (0.5,-2.9)--(0,-3.8);
\draw (0,-3.8+.4)--(0.5,-4.7+.4);

\node at (0.25, -4.7+.4) {$\vdots$};

\draw (0,-5.1+.4)--(0.5,-6+.4);

\end{scope}

\end{tikzpicture}
\caption{Find an embedding of $L$}
\label{constructL}
\end{figure}


Before we examine the effect of replacing $L$ with its minimal symplectic filling $W_L$, we briefly review the classification of minimal symplectic fillings of lens space which can be found in ~\cite{BOn}, ~\cite{Lis}. For notational convenience we denote a linear plumbing graph and a lens space determined the plumbing graph by the same $L$. For a lens space $L$ given by $[b_1, b_2, \dots, b_r]$ , we can choose a concave cap $K_{L}$ of the form
\begin{tikzpicture}[scale=1.1]
\node[bullet] at (-1,0) {};
\node[bullet] at (0,0) {};
\node[bullet] at (1,0) {};
\node[bullet] at (2.5,0) {};
\draw[dotted] (1.5,0)--(2,0);
\draw (-1,0)--(1.5,0);
\draw (2,0)--(2.5,0);
\node[above] at (-1,0) {$+1$};
\node[above] at (0,0) {$-a_{1}\!\!+\!\!1$};
\node[above] at (1,0) {$-a_{2}$};
\node[above] at (2.5,0) {$-a_{n}$};
\end{tikzpicture}
, where $[a_1, a_2, \dots, a_n]$ is a dual continued fraction of $[b_1, b_2, \dots, b_r]$. 
Suppose $X_L\cong \mathbb{CP}^2 \sharp N_0 \overline{\mathbb{CP}^2}$ is a rational symplectic $4$-manifold obtained by gluing two plumbings according to $L$ and $K_{L}$ whose second homology class is generated by $\{l \} \cup E=\{E_1,\dots, E_{N_0}\}$.
Then, for a given minimal symplectic filling $W_L$ of $L$, we get a rational symplectic $4$-manifold $X_{W_{L}}\cong\mathbb{CP}^2 \sharp N \overline{\mathbb{CP}^2}$ by gluing $W_{L}$ and  $K_{L}$ and the image of $K_{L}$ under blowing-downs is isotopic to two complex lines in $\mathbb{CP}^2$, which means that a minimal symplectic filling of $L$ is determined by a homological data of $K_{L}$ in $\mathbb{CP}^2 \sharp N \overline{\mathbb{CP}^2}$ for some $N$. Hence, we draw a curve configuration $C_{W_{L}}$ for $W_{L}$ starting from a configuration of two $(+1)$ strands in $\mathbb{CP}^2$ by blowing-ups with only one $(+1)$ strand. 
This observation shows that the effect of replacing $L$ in $X_L$ with $W_{L}$ is the following: We have another rational symplectic $4$-manifold $X_{W_{L}}\cong\mathbb{CP}^2 \sharp N \overline{\mathbb{CP}^2}$ and the second homology classes in the complement of $L$ are changed so that

 \begin{align*}
l&\rightarrow l\\
[L_i]^E &\rightarrow [L_i]^e \phantom{0} (1\leq i \leq n).\\ 
\end{align*}

where $[L_i]^E$ and $[L_i]^e$ are homology classes of irreducible components of $K_L$ in terms of $\{l \} \cup E=\{E_1,\dots, E_{N_0}\}$ and $\{l \} \cup e=\{e_1,\dots, e_N\}$ respectively.

Let $[C^i_j]^{C}$ and $[C^i_j]^{C'}$ be homology classes of $C^i_j$ in $C$ and $C'$ respectively. Note that $C$ is a curve configuration completed from the last $(-a'_{in'_i})$ strand in the $i^{\text{th}}$-arm of $K'$
by blowing-ups without using any other strand in $C'$. 
If we blow up in the same ways starting with a single $(+1)$ strand instead of $(-a'_{in'_i})$ strand, we get a curve configuration $C_{W_{L}}$ containing $K_{L}$. Hence there is a minimal symplectic filling $W_L$ of $L$ whose homological data of $K_{L}$ in $X_{W_L}(=W_L\cup K_L) \cong \mathbb{CP}^2 \sharp N\overline{\mathbb{CP}^2}$ are given by $[L_j]=[C^i_{n'_i+j-1}]^{C} $ except for $[L_0]=l$ and $[L_1]= l+[C^i_{n'_i}]^{C}-[C^i_{n'_i}]^{C'}$, where $e=\{e_1,\dots, e_N\}$ is homology classes of exceptional spheres coming from the blowing-ups from $C'$ to $C$. 

Finally, we show that the minimal symplectic filling  $W$ corresponding to $C$ is given by $(\widetilde{W}\setminus L) \cup W_L$.
Suppose $X'$ is a rational symplectic $4$-manifold obtained by blowing-ups from a complex line arrangement so that it contains $C'$. 
We take a small Darboux neighborhood $B'$ of a disk $D$ in $C^i_{n'_i}$ of $K'$ so that $B'$ is disjoint from any other irreducible components of $K'$. Now we arrange all the blowing-ups from $C'$ to $C$ inside $B'$ and let $B$ be blowing-ups of $B'$. Then we have a symplectic embedding of $K$ in 
$X=(X' \setminus B') \cup B$ and a homological data of $K$ agrees with $C$. Furthermore, $B \setminus K$ is symplectic deformation equivalent to $W_L$: Consider two complex lines 
in $\mathbb{CP}^2$ and a symplectic embedding of $B'$ such that the image of $D$ in $C^i_{n'_i}$ is a disk in one complex line and $B'$ is disjoint from the other complex line. 
By the construction of $B$, there is a symplectic embedding of $K_L$ in $(\mathbb{CP}^2 \setminus B') \cup B$, where the first component of $K_L$ is the complex line in $(\mathbb{CP}^2 \setminus B')$ and the complement of $K_L$ in $(\mathbb{CP}^2 \setminus B') \cup B$ is symplectic deformation equivalent to $W_L$. Since the complement of a neighborhood of $\mathbb{CP}^1$ in $\mathbb{CP}^2$ is a ball, $B \setminus K = B \setminus K_L$ is also symplectic deformation equivalent to $W_L$. Note that $K=(K\cap (X'\setminus B')) \cup (K \cap B)=(K' \setminus B') \cup (K \cap B)$. Hence we have
$$W=X\setminus K=((X'\setminus B')\setminus K)\cup (B\setminus K)\cong (X'\setminus(K'\cup B')) \cup W_L$$
By a similar argument, $\widetilde{W}\cong (X'\setminus(K'\cup B')) \cup L$, so that $W$ is obtained from $\widetilde{W}$ by replacing $L$ by $W_L$.
\end{proof}
Assume furthermore that there is a $(-1)$ curve intersecting both $C^{i}_{n_i'}$ and another irreducible component $C^k_{l}$ of $K'$ in $C'$.
Then, there is a slight modification of the Lemma~\ref{fundamentallem}, involving two arms of $K$.

\begin{lem}
Suppose that there is a $(-1)$ curve $E$ intersecting $C^{i}_{n_i'}$ and $C^{k}_{l}$ of $K'$ in $C'$ with $a'_{kl} < a_{kl}$. If the standard blowing-ups $\widetilde{C'}$ of $C'$ differs from $C$ only in $C^k_{l}$ and components $C^i_{j}$ for $n'_{i} \leq
j \leq n_i$, then there is a symplectically embedded linear plumbing $L\subset W$ described in Figure~\ref{lem4.2} such that $W$ is obtained by  $\widetilde{W}$ by replacing the plumbing $L$ with some minimal filling $W_L$. Furthermore,
$[b_1, b_2, \dots, b_r]$ is the dual of $[(a_{in'_i}-a'_{in'_i})+1, a_{in'_{i}+1}, a_{in'_{i}+2}, \dots a_{in_i}]$, where $-a_{ij}$ and $-a'_{ij}$ are the weights of $j^{\text{th}}$-component in the $i^{\text{th}}$-arm of $K$ and $K'$ respectively.
\label{fundamentallem3}
\end{lem}
\begin{figure}[htbp]
\begin{tikzpicture}[scale=0.9]
\begin{scope}[scale=0.6]
\node[bullet] at (0,1.5){};

\node[bullet] at (-2,0){};
\node[bullet] at (2,0){};

\node[bullet] at (-2,-1){};
\node[bullet] at (2,-1){};

\node[bullet] at (-2,-3){};
\node[bullet] at (2,-3){};

\draw (0,1.5)--(2,0)--(2,-1.5);
\draw (0,1.5)--(-2,0)--(-2,-1.5);

\draw[dotted](2,-1.5)--(2,-2.5);
\draw[dotted](-2,-1.5)--(-2,-2.5);

\draw(2,-2.5)--(2,-3);
\draw(-2,-2.5)--(-2,-3);
\draw (0,1.5) node[above] {$-b_{1}-1$};

\draw (-2,0) node[right] {$-2$};
\draw (-2,-1) node[right] {$-2$};
\draw (-2,-3) node[right] {$-2$};
\draw (2,0) node[right] {$-b_{2}$};
\draw (2,-1) node[right] {$-b_{3}$};
\draw (2,-3) node[right] {$-b_{r}$};
	\draw [decorate,decoration={brace,mirror, amplitude=5pt},xshift=-10pt,yshift=0pt]
	(-2,0) -- (-2,-3) node [black,midway,xshift=-37pt] 
	{$a_{kl}-a_{kl}'-1$};
\end{scope}

\begin{scope}[shift={(4,1)}]
\draw(0,0)--(2,0);
\draw (0.5,0.2)--(0,-0.7);
\draw (1.5,0.2)--(1,-0.7);

\draw (1,-0.3)--(1.5,-1.2);

\node at (1.25, -1.5) {$\vdots$};

\draw (1,-2)--(1.5,-2.9);

\node[left] at (0,0) {$+1$};
\node[below] at (-0.25, -0.7) {$a_{kl}'-a_{kl}$};
\node[right] at (1.35, -0.35) {$a_{in_i'}'-a_{in_i'}$};

\node[right] at (1.3, -0.95) {$-a_{in_i'+1}$};

\node[right] at (1.25, -2.45) {$-a_{in_i}$};

\end{scope}
\end{tikzpicture}
\caption{A plumbing graph of $L$ and its concave cap $K_L$}
\label{lem4.2}
\end{figure}

\begin{proof}
A proof is similar to that of Lemma~\ref{fundamentallem} except for blowing-ups at two intersection points of $E$ in $C'$  to find an embedding $L$.
That is, we construct a configuration $S$ of strands containing $K$ as in Figure~\ref{detail4.2} whose homological data is equal to that of $\widetilde{C'}$, so that there is a symplectic embedding of $L$ in $\widetilde{W}$.

\begin{figure}[h]
\begin{tikzpicture}[scale=.9]
\begin{scope}

\draw (0.4,0.5)--(0.4,-1.5);

\draw (1.6,0.5)--(1.6,-1.5);
\draw[red,thick, dashdotted] (0,-0.5)--(2,-0.5);
\node[left,red] at (0,-0.5) {$-1$};
\node[red,right] at (2,-0.5) {$E$};

\node[bullet,red] at (0.4,-.5){};

\node[above] at (0.4,0.5) {$-a_{kl}'$};

\node[above] at (1.6,0.4) {$-a_{in_{i}'}'$};

\draw[very thick, ->, >=stealth] (3,-0.65)--(4,-0.65);
\end{scope}
\begin{scope}[shift={(5,0)}]
\draw (0.4,0.5)--(0.4,-1.5);

\draw (3.8,0.5)--(3.8,-1.5);
\node[left,red] at (0.2,-0.8) {$-1$};
\draw[red,thick, dashdotted] (0.2,-0.8)--(1.1,-0.2);

\draw (0.9,-0.2)--(1.8,-0.8);
\node[below] at (1.25,-.5){$-2$};
\draw (1.6,-0.8)--(2.5,-0.2);
\node[below] at (2.25,-.5){$-2$};
\node at (2.8,-.5){$\cdots$};
\draw (3.1,-0.2)--(4,-0.8);
\node[below] at (3.45,-.5){$-2$};

\node[bullet,red] at (3.8,-0.68){};

\node[above] at (0.4,0.5) {$-a_{kl}$};

\node[above] at (3.8,0.4) {$-a_{in_{i}'}'$};

\draw [decorate,decoration={brace, amplitude=5pt},xshift=0pt,yshift=3pt]
	(1.3,-0.2) -- (3.7,-.2) node [black,midway,yshift=12pt] 
	{$a_{kl}-a_{kl}'$};
\draw[very thick, ->, >=stealth] (2.1,-1.8)--(1.1,-2.6);

\end{scope}
\begin{scope}[shift={(-.5,-3.3)}]
\draw (0.4,0.5)--(0.4,-1.5);

\draw (9.6,0.5)--(9.6,-1.5);
\draw[red,thick, dashdotted] (0.2,-0.8)--(1.1,-0.2);

\draw (0.9,-0.2)--(1.8,-0.8);
\node[below] at (1.25,-.5){$-2$};
\draw (1.6,-0.8)--(2.5,-0.2);
\node[below] at (2.25,-.5){$-2$};

\node at (2.8,-.5){$\cdots$};
\draw (3.1,-0.2)--(4,-0.8);
\node[below] at (3.45,-.5){$-2$};

\draw (3.8,-0.8)--(4.7,-0.2);
\node[above] at (4.7,-.2){$-b_{1}-1$};

\draw (4.5,-0.2)--(5.4,-0.8);
\node[below] at (4.95,-.5){$-b_{2}$};
\node at (5.7,-.5){$\cdots$};
\draw (6,-.8)--(6.9,-.2);
\node[below] at (6.55,-.5){$-b_{r}$};

\draw[red,thick, dashdotted] (6.7,-.2)--(7.6,-.8);

\draw (7.4,-.8)--(8.3,-.2);
\node[below] at (8.1,-.5){$-a_{in_i}$};

\node at (8.6,-.5){$\cdots$};

\draw (8.9,-.8)--(9.8,-.2);
\node[below] at (10.5,-0.4) {$-a_{in_i'+1}$};

\node[above] at (0.4,0.5) {$-a_{kl}$};

\node[above] at (9.6,0.5) {$-a_{in_i'}$};

\draw [decorate,decoration={brace, amplitude=5pt},xshift=0pt,yshift=3pt]
	(1.3,-0.2) -- (3.7,-.2) node [black,midway,yshift=12pt] 
	{$a_{kl}-a_{kl}'-1$};

\end{scope}

\end{tikzpicture}

\caption{Embedding of $L$ to $\widetilde{W}$}
\label{detail4.2}
\end{figure}
Next, by viewing $L$ as a two-legged plumbing graph with a degree $(-b_{1}-1)$ of a central vertex, we get a concave cap $K_L$ as in Figure~\ref{lem4.2}: Starting from zero section and infinity section together with two generic fibers of $\mathbb{F}_{b_{1}-1}$, we construct arms corresponding to $[-2,\dots,-2]$ and $[-b_{2},\dots,-b_{r}]$. Then, by blowing-ups at intersection points consecutively of the proper transform of zero section and the arm corresponding to $[-b_{2},\dots,-b_{r}]$, we get a concave cap $K_L$ for $L$. As before, for a given minimal symplectic filling of $L$, we get a rational symplectic $4$-manifold by gluing $K_L$ along $L$ and the image of $K_L$ in $\mathbb{CP}^2$ under blowing-downs is three complex lines in $\mathbb{CP}^2$ intersecting generically, implying that any curve configuration for $K_L$ is obtained from blowing up of an intersection point between two complex lines in $\mathbb{CP}^2$.
Therefore, using blowing-up data from $C'$ to $C$ (Figure~\ref{detail4.21} and Figure~\ref{detail4.22}), we get a minimal symplectic filling $W_L$ of $L$.

\begin{figure}[htbp]

\begin{tikzpicture}[scale=0.8]
\begin{scope}[scale=1]

\draw (0.4,0.5)--(0.4,-1.5);

\draw (1.6,0.5)--(1.6,-1.5);
\draw[red,thick, dashdotted] (0,-0.5)--(2,-0.5);
\node[left,red] at (0,-0.5) {$-1$};

\node[above] at (0.4,0.5) {$-a_{kl}'$};

\node[above] at (1.6,0.4) {$-a_{in_i'}'$};

\draw[very thick, ->, >=stealth] (3,-0.65)--(4,-0.65);
\end{scope}
\begin{scope}[shift={(5.7,0.9)}]

\draw (0.3,0.2)--(-0.2,-0.7);
\draw (1.5,0.2)--(1,-0.7);

\draw (1,-0.3)--(1.5,-1.2);

\draw[red,thick, dashdotted] (-0.2,-0.4)--(1.5,-1.7);

\node at (1.25, -1) {$\vdots$};
\node at (1.25, -2) {$\vdots$};

\draw (1,-1.1)--(1.5,-2.1);

\draw (1,-2.3)--(1.5,-3.2);

\node[right] at (0.05, -0.25) {$-a_{kl}$};
\node[right] at (1.25, -0.25) {$-a_{in_i'}$};

\node[right] at (1.25, -0.75) {$-a_{in_i'+1}$};

\node[right] at (1.25, -2.75) {$-a_{in_i}$};

\end{scope}
\end{tikzpicture}

\caption{blowing-ups from $C'$ to $C$}
\label{detail4.21}
\end{figure}
\begin{figure}[htbp]
\begin{tikzpicture}[scale=0.8]
\begin{scope}[scale=1]

\draw (0,0.2)--(2,0.2);
\draw (0.4,0.5)--(0.4,-1.5);

\draw (1.6,0.5)--(1.6,-1.5);
\draw[red,thick, dashdotted] (0,-0.6)--(2,-0.6);
\node[left,red] at (0,-0.6) {$-1$};
\node[left] at (0,0.2) {$+1$};

\node[above] at (0.4,0.5) {$0$};

\node[above] at (1.6,0.5) {$0$};

\draw[very thick, ->, >=stealth] (3,-0.65)--(4,-0.65);
\end{scope}
\begin{scope}[shift={(5.5,0.9)}]
\draw(0,0)--(2,0);
\draw (0.5,0.2)--(0,-0.7);
\draw (1.5,0.2)--(1,-0.7);

\draw (1,-0.3)--(1.5,-1.2);

\draw[red,thick, dashdotted] (0,-0.4)--(1.5,-1.7);

\node at (1.25, -1) {$\vdots$};
\node at (1.25, -2) {$\vdots$};

\draw (1,-1.1)--(1.5,-2.1);

\draw (1,-2.3)--(1.5,-3.2);

\node[left] at (0,0) {$+1$};
\node[below] at (-0.25, -0.7) {$a_{kl}'-a_{kl}$};
\node[right] at (1.25, -0.35) {$a_{in_i'}'-a_{in_i'}$};

\node[right] at (1.25, -.9) {$-a_{in_i'+1}$};

\node[right] at (1.25, -2.75) {$-a_{in_i}$};

\end{scope}
\end{tikzpicture}
\caption{Curve configuration for $W_L$}
\label{detail4.22}
\end{figure}
Suppose that $X'$ is a rational symplectic $4$-manifold containing $C''$, obtained from $C'$ by blowing down $E$, and let $B'$ be a small Darboux neighborhood of the intersection point coming from the blowing down. Then a similar argument as in Lemma~\ref{fundamentallem} above shows that the minimal symplectic filling $W$ corresponding to $C$ is obtained from $\widetilde{W
}$ by replacing $L$ with $W_L$.
\end{proof}
Assume that $C'$ is a curve configuration containing $K'\leq K$ corresponding to a minimal symplectic filling $W'$ of another small Seifert $3$-manifold $Y'$ and $\widetilde{C'}$ is a curve configuration obtained from $C'$ by standard blowing-ups. Then we can describe a minimal symplectic filling $\widetilde{W}$ of $Y$ corresponding to $\widetilde{C'}$ explicitly. 
\begin{lem}
Under the assumption above, there is a symplectically embedded plumbing of $2$-spheres $\Gamma'$ in the minimal resolution $\widetilde{M}$ so that a minimal symplectic filling $\widetilde{W}$ of $Y$ corresponding to $\widetilde{C'}$ is obtained from $\widetilde{M}$ by replacing $\Gamma'$ with $W'$.
\label{fundamentallem2}
\end{lem}

\begin{proof}
Let $K_0$ be a plumbing graph determined by black strands in $(a)$-Figure~\ref{blowingup1}. Clearly, $K_0\leq K$ so that there is a curve configuration $C_{\widetilde{M}}$ obtained by standard blowing-ups from $(a)$. We first show that the curve configuration $C_{\widetilde{M}}$ corresponds to the minimal resolution $\widetilde{M}$. 
Recall that a concave cap $K$ in Figure~\ref{cap} can be found in ~\cite{SSW} and ~\cite{Sta1}: Starting from the zero and infinity sections with $(b-1)$ generic fibers of a Hirzebruch surface $\mathbb{F}_1$ which can be drawn as $(a)$ in Figure~\ref{blowingup1}, we blow up intersection points of generic fibers and the infinity section so that we have a $(-b)$ rational curve which corresponds to the central vertex of the minimal resolution graph $\Gamma$.
Then, we obtain a linear chain of strands containing both $i^{\text{th}}$-arm of $K$ and $\Gamma$ from two $(-1)$ strands by blowing-ups as in Figure~\ref{constructarm}.
As a result, we have a configuration $S_{\widetilde{M}}$ containing both $\Gamma$ and $K$ disjointly, so that the complement of $K$ in a rational surface $X_Y$ is the minimal resolution $\widetilde{M}$ and $K$ is a concave cap for $Y$. By using the same argument as in the proof of Lemma~\ref{fundamentallem} above, we conclude that $C_{\widetilde{M}}$ is a curve configuration for $\widetilde{M}$.

\begin{figure}[h]

\begin{tikzpicture}[xscale=0.9]

\begin{scope}
\draw(0.2,0)--(4,0);
\draw (0.6,0.2)--(0.6,-2.6);
\draw (1.2,0.2)--(1.2,-2.6);
\draw (1.8,0.2)--(1.8,-2.6);
\draw (2.4,0.2)--(2.4,-2.6);
\draw (3.6,0.2)--(3.6,-2.6);

\draw (0.2,-2.4)--(4,-2.4);

\node[left] at (0.2,0) {$+1$};
\node[left] at (0.6,-1.2){$0$};
\node[left] at (1.2,-1.2){$0$};
\node[left] at (1.8,-1.2){$0$};
\node[left] at (2.4,-1.2){$0$};

\node at (2.9,-1.2){$\cdots$};
\node[left] at (3.6,-1.2){$0$};

\node at (4.45,-1.2){$\rightarrow$};
\end{scope}

\begin{scope}[shift={(4.25+1.25,0)}]
\draw(0,0)--(4.4,0);

\draw (0,-2.4)--(4.4,-2.4);

\draw (0,-1.4)--(0.5,0.2);
\draw (0,-1.)--(0.5,-2.6);

\draw (0.75,-1.4)--(1.25,0.2);
\draw (0.75,-1.)--(1.25,-2.6);

\draw (1.5,-1.4)--(2,0.2);
\draw (1.5,-1.)--(2,-2.6);

\draw (2.25,-1.4)--(2.75,0.2);
\draw (2.25,-1.)--(2.75,-2.6);

\draw (2.25+1.3,-1.4)--(2.75+1.3,0.2);
\draw (2.25+1.3,-1.)--(2.75+1.3,-2.6);

\node[left] at (0,0) {$+1$};
\node[left] at (0,-2.4) {$-b$};
\node at (0.55,-0.6){$-1$};
\node at (1.3,-0.6){$-1$};
\node at (2.05,-0.6){$-1$};
\node at (2.8,-0.6){$-1$};
\node at (2.8+1.3,-0.6){$-1$};
\node at (3.1,-1.2){$\cdots$};
\node at (0.55,-1.8){$-1$};
\node at (1.3,-1.8){$-1$};
\node at (2.05,-1.8){$-1$};
\node at (2.8,-1.8){$-1$};

\node at (2.8+1.3,-1.8){$-1$};
\end{scope}

\begin{scope}[shift={(9,0)}]
\end{scope}

\end{tikzpicture}
\caption{Blowing up Hirzebruch surface $\mathbb{F}_1$}
\label{constructK}
\end{figure}

\begin{figure}[h]
\begin{tikzpicture}[xscale=0.9]
\begin{scope}

\draw (0,-1.4)--(0.5,0.2);
\draw (0,-1.)--(0.5,-2.6);
\node at (0.55,-0.6){$-1$};
\node at (0.55,-1.8){$-1$};
\filldraw[red] (1/16,-1.2) circle (1.2pt);

\draw[-latex,line width=1.5pt] (2.9-1.,-1.2)--(4.15-1.,-1.2) ;

\end{scope}

\begin{scope}[shift={(4.25+1.25-1,1.1)},scale=.8]

\node[right] at (0.25, -0.25) {$-a_{i1}$};
\node[right] at (0.25, -0.9) {$-a_{i2}$};
\node[right] at (0.25, -2) {$-a_{in_1}$};
x
\node[right] at (0.25, -3.35) {$-b_{ir_i}$};
\node[right] at (0.25, -4) {$-b_{ir_i-1}$};
\node[right] at (0.25, -5.2) {$-b_{i1}$};

\draw (0.5,0.2)--(0,-0.7);

\draw (0,-0.3)--(0.5,-1.2);

\node at (0.25, -1.2) {$\vdots$};

\draw (0.5,-1.6)--(0.,-2.5);
\draw[red,thick, dashdotted] (0.,-2.25)--(0.5,-3.15);
\draw (0.5,-2.9)--(0,-3.8);
\draw (0,-3.8+.4)--(0.5,-4.7+.4);

\node at (0.25, -4.7+.4) {$\vdots$};

\draw (0,-5.1+.4)--(0.5,-6+.4);

\end{scope}
\end{tikzpicture}
\caption{Construction of each arm in $K$ and $\Gamma$}
\label{constructarm}
\end{figure}

In the same way, we could get a configuration $S_{\Gamma'}$ of strands containing both $K'$ and a plumbing graph $\Gamma'$ so that the complement of $K'$ in the resulting rational symplectic $4$-manifold $X_{Y'}\cong \mathbb{CP}^2 \sharp M \overline{\mathbb{CP}^2}$ is a plumbing of $2$-spheres according to $\Gamma'$. Note that $M\!-\!1$ is the number of blowing-ups in the standard blowing-ups from $(a)$ in Figure~\ref{blowingup1} to $K'$.
Since $K'\leq K$, we obtain a configuration $S'_{\widetilde{M}}$ of strands containing $\Gamma'$ and $K$ disjointly from $S_{\Gamma'}$ by standard blowing-ups at non-intersection point in the last component of each $i^{\text{th}}$-arm of $K'$. 
Let $X = X_{Y'} \sharp N \overline{\mathbb{CP}^2}$ be a resulting rational symplectic 4-manifold. 
Then $X\cong X_Y=\widetilde{M} \cup K$ because the number of blowing-ups in the standard blowing-ups from $(a)$ in Figure~\ref{blowingup1} to $K$ is equal to a sum of numbers of blowing-ups for $(a)$ in Figure~\ref{blowingup1} to $K'$ and $K'$ to $K$.
Furthermore, a homological data of $K$ in $S'_{\widetilde{M}}$ is also equal to that of $C_{\widetilde{M}}$. Hence  a plumbing graph $\Gamma'$  is symplectically embedded in $\widetilde{M}$.



If there is a sequence of blowing-ups from a configuration of strands representing $K'$ to $K$, then we have a corresponding symplectic cobordism $Z$ from $Y'$ to $Y$ because the total transform of $K'$ is still a concave cap for $Y'$ while $K$ is a concave cap for $Y$. In particular, if $C'$ is a curve configuration for a minimal symplectic filling $W'$ then we get a curve configuration $C$ for a minimal filling $W$ by the sequence of blowing-ups from $K'$ to $K$ and $W=W' \cup Z$. In case of standard blowing-ups from $K'$ to $K$, we can deduce from the construction of $S'_{\widetilde{M}}$  that the corresponding cobordism is equal to $\widetilde{M} \setminus \Gamma'$. Hence we have $\widetilde{W}= W' \cup (\widetilde{M}\setminus \Gamma')$, so that $\widetilde{W}$ is obtained from $\widetilde{M}$ by replacing $\Gamma'$ with $W'$. 
\end{proof}

\subsection{Proof for type A}

For a curve configuration $C$ of type A, we want to show that the corresponding minimal symplectic filling $W$ is obtained from the minimal resolution $\widetilde{M}$ by replacing each arm in the resolution graph $\Gamma$ with its minimal symplectic symplectic filling. 
Since we already know in the proof of Lemma~\ref{fundamentallem2} above that a curve configuration $C_{\widetilde{M}}$ ,which is obtained from (a) in Figure~\ref{blowingup1} by standard blowing-ups, corresponds to $\widetilde{M}$, by applying repeatedly Lemma~\ref{fundamentallem} with $K'$ as in $(a)$ in Figure~\ref{blowingup1} so that the corresponding $L$ is  one of three arms in $\Gamma$, we conclude that all minimal symplectic fillings corresponding to a curve configuration $C$ of type A are obtained by a sequence of rational blowdowns from the minimal resolution $\widetilde{M}$.\\

The following example illustrates this case.

\begin{example}
Let $Y$ be a small Seifert $3$-manifold whose associated plumbing graph and concave cap are shown in Figure~\ref{b5plumbing}. Then, there are two curve configurations of type A as in Figure~\ref{curveconfiguration1}. 
Of course, there exist other curve configurations of type B and C for minimal symplectic fillings of $Y$, which will be treated in Example~\ref{example2} and Example~\ref{example3} later.
Note that each dash-dotted strand represents an exceptional $2$-sphere, that is, a $2$-sphere with self-intersection $-1$. 
We omit the degree of irreducible components of the concave cap for the sake of convenience in the figure. 
The left-hand curve configuration in Figure~\ref{curveconfiguration1} is obtained by standard blowing-ups from that of Figure~\ref{blowingup1} which means that the corresponding minimal filling is the minimal resolution. 
Note that only the third arm in the plumbing graph $\Gamma$ has a nontrivial minimal symplectic filling that is obtained by rationally blowing down the $(-4)$ 2-sphere. Using Lisca's description of the minimal symplectic fillings of lens spaces, we obtain the right-hand curve configuration in Figure~\ref{curveconfiguration1}, which represents a minimal symplectic filling obtained from the minimal resolution by rationally blowing down the $(-4)$ 2-sphere in the third arm.
\begin{figure}[h]
\begin{tikzpicture}[scale=0.9]
\begin{scope}
\node[bullet] at (2,0){};
\node[bullet] at (3,0){};
\node[bullet] at (4,0){};
\node[bullet] at (5,0){};
\node[bullet] at (6,0){};

\node[bullet] at (4,-1){};

\node[above] at (2,0){$-3$};
\node[above] at (3,0){$-2$};
\node[above] at (4,0){$-5$};
\node[above] at (5,0){$-4$};
\node[above] at (6,0){$-2$};

\node[left] at (4,-1){$-2$};

\draw (2,0)--(6,0);
\draw (4,0)--(4,-1);
\end{scope}
\begin{scope}[shift={(8,0.5)}]
\draw(0,0)--(4,0);
\draw (0.5,0.2)--(0,-1.);
\draw (1.5,0.2)--(1,-1.);
\draw (2.5,0.2)--(2,-1.);
\draw (3.5,0.2)--(3,-1.);

\draw (0,-0.6)--(0.5,-1.8);

\draw (2,-0.6)--(2.5,-1.8);
\draw (2.5, -1.4)--(2,-2.6);

\node[left] at (0,0) {$+1$};
\node[right] at (0.25, -0.25) {$-3$};
\node[right] at (1.25, -0.25) {$-2$};
\node[right] at (2.25, -0.25) {$-2$};
\node[right] at (3.25, -0.25) {$-1$};

\node[right] at (0.25, -1) {$-2$};

\node[right] at (2.25, -1) {$-2$};
\node[right] at (2.25, -2.25) {$-3$};

\end{scope}
\end{tikzpicture}
\caption{Plumbing graph $\Gamma$ and its concave cap $K$}
\label{b5plumbing}
\end{figure}
\label{b5example}

\begin{figure}[h]
\begin{tikzpicture}[scale=0.9]
\begin{scope}
\draw(0,0.05)--(4,0.05);
\draw[red,thick, dashdotted] (0,-0.1)--(4,-0.1);

\draw (0.5,0.2)--(0,-1.);
\draw[red,thick, dashdotted] (0.1,-0.25)--(0.5,-0.25);
\draw[red,thick, dashdotted] (0.05,-0.35)--(0.45,-0.35);

\draw (1.5,0.2)--(1,-1.);
\draw[red,thick, dashdotted] (1.1,-0.25)--(1.5,-0.25);
\draw[red,thick, dashdotted] (1.05,-0.35)--(1.45,-0.35);

\draw (2.5,0.2)--(2,-1.);
\draw[red,thick, dashdotted] (2.1,-0.3)--(2.5,-0.3);

\draw (3.5,0.2)--(3,-1.);
\draw[red,thick, dashdotted] (3.1,-0.3)--(3.5,-0.3);

\draw (0,-0.6)--(0.5,-1.8);
\draw[red,thick, dashdotted] (0.1,-1.3)--(0.5,-1.3);

\draw (2,-0.6)--(2.5,-1.8);
\draw (2.5, -1.4)--(2,-2.6);
\draw[red,thick, dashdotted] (2.1,-2)--(2.5,-2);
\draw[red,thick, dashdotted] (2.0,-2.1)--(2.4,-2.1);

\node[left] at (0,0.05) {$+1$};



\end{scope}
\begin{scope}[shift={(6,0)}]
\draw(0,0.05)--(4,0.05);
\draw[red,thick, dashdotted] (0,-0.1)--(4,-0.1);

\draw (0.5,0.2)--(0,-1.);
\draw[red,thick, dashdotted] (0.1,-0.25)--(0.5,-0.25);
\draw[red,thick, dashdotted] (0.05,-0.35)--(0.45,-0.35);

\draw (1.5,0.2)--(1,-1.);
\draw[red,thick, dashdotted] (1.1,-0.25)--(1.5,-0.25);
\draw[red,thick, dashdotted] (1.05,-0.35)--(1.45,-0.35);

\draw (2.5,0.2)--(2,-1.);

\draw (3.5,0.2)--(3,-1.);
\draw[red,thick, dashdotted] (3.1,-0.3)--(3.5,-0.3);

\draw (0,-0.6)--(0.5,-1.8);
\draw[red,thick, dashdotted] (0.1,-1.3)--(0.5,-1.3);

\draw (2,-0.6)--(2.5,-1.8);
\draw[red,thick, dashdotted] (2.05,-1.2)--(2.45,-1.2);
\draw (2.5, -1.4)--(2,-2.6);
\draw[red,thick, dashdotted] (2.05,-2.05)--(2.45,-2.05);

\node[left] at (0,0.05) {$+1$};

\end{scope}
\end{tikzpicture}
\caption{Two curve configurations in Example~\ref{example1}}
\label{curveconfiguration1}
\end{figure}
\label{example1}
\end{example}\smallskip

\subsection{Proof for type B}
For a curve configuration $C$ of type B, we want to show that the corresponding minimal symplectic filling $W$ is obtained from the minimal resolution $\widetilde{M}$ by replacing disjoint subgraphs in the resolution graph $\Gamma$ with their minimal symplectic symplectic filling. By reindexing if needed, we assume that the first and the second arm of the configurations in Figure~\ref{startingposition} becomes 
the first and the second arm of $K$ in $C$, respectively,
and the proper transform of $e_2$ is not an irreducible component of $K$. 
Since we do not use $e_2$ during the blowing-ups, we can get the first and the second arm of $K$, so that the homological data for the irreducible components in these arms agrees
with that of $C$, from the configurations in Figure~\ref{startingposition} leaving the third single $(-1)$ arm unchanged.
Hence we arrange the order of blowing-ups from a configuration in Figure~\ref{startingposition} to $C$ so that we have an intermediate configuration $C'$ of strands containing $K'\leq K$ as in Figure~\ref{subcap}.
Note that the degree of strands in $C' \setminus K'$ is all $-1$.
\begin{figure}[h]
\begin{center}
\begin{tikzpicture}[scale=1.2]
\draw(0,0)--(5,0);
\draw (0.5,0.2)--(0,-0.7);
\draw (1.5,0.2)--(1,-0.7);
\draw (2.5,0.2)--(2,-0.7);
\draw (3.5,0.2)--(3,-0.7);
\draw (4.5,0.2)--(4,-0.7);
\node at (4,0.1) {$\cdots$};

	\draw [decorate,decoration={brace,amplitude=5pt},xshift=0pt,yshift=3pt]
	(2.5,0.2) -- (4.5,0.2) node [black,midway,yshift=8pt] 
	{\footnotesize $b-3$};

\draw (0,-0.3)--(0.5,-1.2);
\draw (1,-0.3)--(1.5,-1.2);

\node at (0.25, -1.5) {$\vdots$};
\node at (1.25, -1.5) {$\vdots$};

\draw (0,-2)--(0.5,-2.9);
\draw (1,-2)--(1.5,-2.9);

\node[left] at (0,0) {$+1$};
\node[right] at (0.25, -0.25) {$-a_{11}$};
\node[right] at (1.25, -0.25) {$-a_{21}$};

\node[right] at (2.25, -0.25) {$-1$};
\node[right] at (3.25, -0.25) {$-1$};
\node[right] at (4.25, -0.25) {$-1$};

\node[right] at (0.25, -0.75) {$-a_{12}$};
\node[right] at (1.25, -0.75) {$-a_{22}$};

\node[right] at (0.25, -2.45) {$-a_{1n_1}$};
\node[right] at (1.25, -2.45) {$-a_{2n_2}$};

\end{tikzpicture}
\end{center}
\caption{Concave cap $K'$ for linear subgraph of $\Gamma$}
\label{subcap}
\end{figure}
If we choose a linear plumbing graph $L'=$
\begin{tikzpicture}[scale=0.7]

\node[bullet] at (1,0) {};
\node[bullet] at (2.5,0) {};
\node[bullet] at (3.5,0) {};
\node[bullet] at (4.5,0) {};
\node[bullet] at (6,0) {};
\draw[dotted] (1.5,0)--(2,0);
\draw (1,0)--(1.5,0);
\draw (2,0)--(2.5,0);
\draw (2.5,0)--(5,0);
\draw (5.5,0)--(6,0);
\draw[dotted] (5,0)--(5.5,0);
\node[above] at (1,0) {$-b_{1r_1}$};
\node[above] at (2.5,0) {$-b_{11}$};
\node[above] at (3.5,0) {$-b$};
\node[above] at (4.5,0) {$-b_{21}$};
\node[above] at (6,0) {$-b_{2r_2}$};
\end{tikzpicture} 
, a subgraph of $\Gamma$ as a two-legged plumbing graph with the $(-b)$ central vertex, then $K'$ gives a concave cap of $L'$ and $C'$ is a curve configuration for a minimal symplectic filling $W_{L'}$ of $L'$.

Let $C_1$ be a curve configuration obtained by standard blowing-ups from $C'$. Then, by Lemma~\ref{fundamentallem2}, the curve configuration $C_1$ corresponds to a minimal symplectic filling $W_1$, which is obtained from the minimal resolution $\widetilde{M}$ by replacing $L'$ with $W_{L'}$. Furthermore, since $[a_{31}, a_{32}, \dots, a_{3n_3}]=[2, \dots, 2, c_1+1, c_2,  \dots, c_k]$, where $[c_1, c_2, \dots, c_k]$ is the dual of $[b_{32}, b_{33}, \dots, b_{3r_3}]$, by Lemma~\ref{fundamentallem} with $L$ as a linear chain determined by $[b_{32}, b_{33}, \dots, b_{3r_3}]$,  we conclude that the minimal symplectic filling $W$ corresponding to $C$ is obtained from $W_1$ by replacing $L$ with its minimal symplectic filling.
Hence the desired minimal symplectic filling $W$ is obtained from $\widetilde{M}$ by replacing disjoint linear subgraphs  
\begin{tikzpicture}[scale=0.8]

\node[bullet] at (1,0) {};
\node[bullet] at (2.5,0) {};
\node[bullet] at (3.5,0) {};
\node[bullet] at (4.5,0) {};
\node[bullet] at (6,0) {};
\draw[dotted] (1.5,0)--(2,0);
\draw (1,0)--(1.5,0);
\draw (2,0)--(2.5,0);
\draw (2.5,0)--(5,0);
\draw (5.5,0)--(6,0);
\draw[dotted] (5,0)--(5.5,0);
\node[above] at (1,0) {$-b_{1r_1}$};
\node[above] at (2.5,0) {$-b_{11}$};
\node[above] at (3.5,0) {$-b$};
\node[above] at (4.5,0) {$-b_{21}$};
\node[above] at (6,0) {$-b_{2r_2}$};
\end{tikzpicture} and \begin{tikzpicture}
\node[bullet] at (0,0) {};
\node[bullet] at (1,0) {};
\node[bullet] at (2.5,0) {};
\draw[dotted] (1.5,0)--(2,0);
\draw (0,0)--(1.5,0);
\draw (2,0)--(2.5,0);
\node[above] at (0,0) {$-b_{32}$};
\node[above] at (1,0) {$-b_{33}$};
\node[above] at (2.5,0) {$-b_{3r_3}$};
\end{tikzpicture}\!\! of $\Gamma$ with their minimal symplectic fillings,
so that there is a sequence of rational blowdowns from $\widetilde{M}$ to $W$.\\

The following example illustrates the curve configurations of type B

\begin{example}
We again consider a small Seifert $3$-manifold $Y$ used in Example~\ref{example1}. Since the left-hand configuration without exceptional $2$-spheres in Figure~\ref{middleposition} gives a concave cap of a lens space determined by a subgraph \begin{tikzpicture}[scale=0.8]
\node[bullet] at (2,0){};
\node[bullet] at (3,0){};
\node[bullet] at (4,0){};
\node[bullet] at (5,0){};

\node[above] at (2,0){$-3$};
\node[above] at (3,0){$-2$};
\node[above] at (4,0){$-5$};
\node[above] at (5,0){$-2$};

\draw (2,0)--(5,0);
\end{tikzpicture} of $\Gamma$, it gives a minimal symplectic filling $W_L$ of the lens space $L(39,16)$.
Then, by blowing-ups at points lying on the third arm different from the intersection point with the exceptional curve $e$, we get an embedding of a concave cap $K$ of $Y$ as in the right-hand curve configuration $C_1$ of Figure~\ref{middleposition}, which gives a minimal symplectic filling $W_1$ of $Y$. Furthermore, since there is a unique minimal symplectic filling of lens space $L(2,1)$ corresponding to the $(-2)$ $2$-sphere in the third arm of $\Gamma$, $W_1$ is obtained from the minimal symplectic filling $W_L$. 
In fact, there are three more minimal symplectic fillings of $Y$ which are of Case B type - See Figure~\ref{3.1type} for the corresponding curve configurations. Note that the curve configuration $C_1$ for $W_1$ in Figure~\ref{middleposition} comes from the right-hand configuration in Figure~\ref{startingposition} and the curve $c$ becomes a component of the first arm of $K$ in Figure~\ref{b5plumbing}. Similarly, the curve configuration $C_i$ for $W_i$ $(2\leq i \leq 4)$ is also obtained from the right-hand configuration in Figure~\ref{startingposition}. One can easily check that each $W_i$ is obtained from the minimal resolution of $Y$ by a linear rational blowdown surgery: Explicitly $W_2$, $W_3$ and $W_4$ are obtained by rationally blowing-down along subgraphs \begin{tikzpicture}[scale=0.8]
\node[bullet] at (2,0){};
\node[bullet] at (3,0){};

\node[above] at (2,0){$-2$};
\node[above] at (3,0){$-5$};

\draw (2,0)--(3,0);
\end{tikzpicture}, \begin{tikzpicture}[scale=0.8]
\node[bullet] at (2,0){};
\node[bullet] at (3,0){};

\node[above] at (2,0){$-5$};
\node[above] at (3,0){$-2$};

\draw (2,0)--(3,0);
\end{tikzpicture} and \begin{tikzpicture}[scale=0.8]
\node[bullet] at (2,0){};
\node[bullet] at (3,0){};
\node[bullet] at (4,0){};
\node[bullet] at (5,0){};
\node[bullet] at (6,0){};

\node[above] at (2,0){$-3$};
\node[above] at (3,0){$-2$};
\node[above] at (4,0){$-5$};
\node[above] at (5,0){$-4$};
\node[above] at (6,0){$-2$};

\draw (2,0)--(6,0);
\end{tikzpicture} in $\Gamma$ respectively. And $W_1$ is also obtained by rationally blowing-down along \begin{tikzpicture}[scale=0.8]
\node[bullet] at (2,0){};
\node[bullet] at (3,0){};
\node[bullet] at (4,0){};

\node[above] at (2,0){$-3$};
\node[above] at (3,0){$-5$};
\node[above] at (4,0){$-2$};

\draw (2,0)--(4,0);
\end{tikzpicture}, which is embedded in another plumbing \begin{tikzpicture}[scale=0.8]
\node[bullet] at (2,0){};
\node[bullet] at (3,0){};
\node[bullet] at (4,0){};
\node[bullet] at (5,0){};

\node[above] at (2,0){$-3$};
\node[above] at (3,0){$-2$};
\node[above] at (4,0){$-5$};
\node[above] at (5,0){$-2$};

\draw (2,0)--(5,0);
\end{tikzpicture}.


\begin{figure}[h]
\begin{tikzpicture}[scale=0.9]
\begin{scope}
\draw (0,0.2)--(2.8,0.2);
\draw (0.4,0.5)--(0.4,-2.4);

\draw (1.4,0.5)--(1.4,-0.6);
\draw (1.9,0.5)--(1.9,-1.3);
\draw (2.4,0.5)--(2.4,-2.4);

\draw (0.3,-0.4)--(1,-0.8);
\draw[red,thick, dashdotted] (0.8,-0.8)--(1.5,-0.4);

\draw[red,thick, dashdotted] (0.2,0)--(0.6,0);
\draw[red,thick, dashdotted] (1.2,0)--(2.6,0);

\draw[red,thick, dashdotted] (0.2,-1.1)--(2.1,-1.1);
\node[left,red] at (0.2,-1.1){$e$};
\draw[red,thick, dashdotted] (0.2,-2.2)--(2.6,-2.2);

\node[above] at (1.9,0.5) {$-1$};
\node[above] at (2.4,0.5) {$-1$};
\node[left] at (0,0.2) {$+1$};

\draw[very thick, ->, >=stealth] (3.3,-0.85)--(4.3,-0.85);

\end{scope}
\begin{scope}[shift={(5,0)}]
\draw (0,0.2)--(2.8,0.2);
\draw (0.4,0.5)--(0.4,-2.4);

\draw (1.4,0.5)--(1.4,-0.6);

\draw (1.9,0.5)--(1.9,-1.5);
\draw (1.8,-1.2)--(2.1,-1.7);
\draw(2.1,-1.6)--(1.8,-2.1);
\draw[red,thick, dashdotted] (1.85,-1.8)--(2.05,-1.8);
\draw[red,thick, dashdotted] (1.8,-1.9)--(2.0,-1.9);

\draw (2.4,0.5)--(2.4,-2.4);

\draw (0.3,-0.4)--(1,-0.8);
\draw[red,thick, dashdotted] (0.8,-0.8)--(1.5,-0.4);

\draw[red,thick, dashdotted] (0.2,0)--(0.6,0);
\draw[red,thick, dashdotted] (1.2,0)--(2.6,0);

\draw[red,thick, dashdotted] (0.2,-1.1)--(2.1,-1.1);
\draw[red,thick, dashdotted] (0.2,-2.2)--(2.6,-2.2);

\node[above] at (2.4,0.5) {$-1$};
\node[left] at (0,0.2) {$+1$};

\end{scope}
\end{tikzpicture}
\caption{Curve configuration $C_1$ for $W_1$}
\label{middleposition}
\end{figure}
\begin{figure}[h]
\begin{tikzpicture}
\begin{scope}
\node at (1.4,-2.8){$C_2$};
\draw (0,0.3)--(2.8,0.3);
\draw (0.4,0.5)--(0.4,-2.4);

\draw (1.4,0.5)--(1.4,-0.6);

\draw (1.9,0.5)--(1.9,-1.5);
\draw (1.8,-1.2)--(2.1,-1.7);
\draw(2.1,-1.6)--(1.8,-2.1);
\draw[red,thick, dashdotted] (1.75,-1.8)--(2.15,-1.8);
\draw[red,thick, dashdotted] (1.7,-1.9)--(2.1,-1.9);

\draw (2.4,0.5)--(2.4,-2.4);

\draw (0.3,-0.5)--(1,-0.9);
\draw[red,thick, dashdotted](0.7,-0.95)--(.9,-0.55);

\draw[red,thick, dashdotted] (0.2,-0.2)--(1.6,-0.2);
\draw[red,thick, dashdotted] (1.2,0)--(2.6,0);
\draw[red,thick, dashdotted] (1.2,-0.4)--(1.6,-0.4);
\draw[red,thick, dashdotted] (0.2,-1.1)--(2.1,-1.1);
\draw[red,thick, dashdotted] (0.2,-2.2)--(2.6,-2.2);

\node[above] at (2.4,0.5) {$-1$};
\node[left] at (0,0.3) {$+1$};

\end{scope}
\begin{scope}[shift={(4,0)}]
\node at (1.4,-2.8){$C_3$};
\draw (0,0.3)--(2.8,0.3);
\draw (0.4,0.5)--(0.4,-2.4);

\draw (1.4,0.5)--(1.4,-0.7);

\draw (1.9,0.5)--(1.9,-1.5);
\draw (1.8,-1.2)--(2.1,-1.7);
\draw(2.1,-1.6)--(1.8,-2.1);
\draw[red,thick, dashdotted] (1.75,-1.8)--(2.15,-1.8);
\draw[red,thick, dashdotted] (1.7,-1.9)--(2.1,-1.9);

\draw (2.4,0.5)--(2.4,-2.4);
\draw (1.3,-0.5)--(1.7,-1.);
\draw[red,thick, dashdotted] (1.35,-.85)--(1.75,-.85);

\draw[red,thick, dashdotted] (1.2,-0.3)--(1.6,-0.3);

\draw[red,thick, dashdotted] (1.2,0.1)--(2.6,0.1);
\draw[red,thick, dashdotted] (.2,-0.1)--(1.6,-0.1);
\draw[red,thick, dashdotted] (0.2,-1.1)--(2.1,-1.1);
\draw[red,thick, dashdotted] (0.2,-2.2)--(2.6,-2.2);

\node[above] at (2.4,0.5) {$-1$};
\node[left] at (0,0.3) {$+1$};

\end{scope}
\begin{scope}[shift={(8,0)}]
\node at (1.4,-2.8){$C_4$};
\draw (0,0.3)--(2.8,0.3);
\draw (0.4,0.5)--(0.4,-2.4);

\draw (1.6,0.5)--(1.6,-0.6);

\draw (2.,0.5)--(2.,-2.0);

\draw (2.4,0.5)--(2.4,-2.4);

\draw (0.3,-0.4)--(0.8,-0.8);
\draw (1.1,-0.4)--(0.6,-0.8);
\draw[red,thick, dashdotted] (0.9,-0.4)--(1.4,-0.8);
\draw (1.2,-0.8)--(1.7,-0.4);


\draw[red,thick, dashdotted] (1.4,0.1)--(2.6,0.1);

\draw[red,thick, dashdotted] (1.8,-1.4)--(2.2,-1.4);

\draw[red,thick, dashdotted] (0.2,-1.8)--(2.2,-1.8);
\draw[red,thick, dashdotted] (0.2,-2.2)--(2.6,-2.2);

\node[above] at (2.4,0.5) {$-1$};
\node[left] at (0,0.3) {$+1$};

\end{scope}

\end{tikzpicture}
\caption{Curve configurations for other symplectic fillings of $Y$}
\label{3.1type}
\end{figure}
\label{example2}
\end{example}\smallskip

\subsection{Proof for type C}
For a minimal symplectic filling $W$ corresponding to a curve configuration $C$ of type C, we want to find a curve configuration $C_1$ of type B such that there is a symplectically embedded linear chain $L$ of $2$-spheres (that is not visible in $\Gamma$) in $W_1$ corresponding to $C_1$  so that $W$ is obtained from $W_1$ by replacing $L$ with its minimal symplectic filling $W_L$.


\begin{figure}[h]
\begin{center}
\begin{tikzpicture}[xscale=1.1,yscale=.75]
\begin{scope}
\draw(0,0)--(3.5,0);
\draw (0.5,0.2)--(0,-0.7);
\draw (1.8,0.2)--(1.3,-0.7);
\draw (3.2,0.2)--(2.7,-2);

\draw (0,-0.3)--(0.5,-1.2);
\draw (1.3,-0.3)--(1.8,-1.2);

\node at (0.25, -1.2) {$\vdots$};

\draw (0.5,-1.6)--(0.,-2.5);

\draw (0.,-2.)--(0.5,-2.9);

\draw (0.5,-3.3)--(0,-4.2);

\draw[red,thick, dashdotted] (.2,-1.7)--(1.35,-1.74);
\draw[red,thick, dashdotted] (3,-1.8)--(1.75,-1.76);
\node[red,below] at (3,-1.8){$e_2$};

\draw (1.8,-3.3)--(1.3,-4.2);

\draw[red,thick, dashdotted] (0.2,-3.4)--(1.8,-3.7);
\node[red,below] at (1.8,-3.7){$e$};
\draw (0,-3.8)--(0.5,-4.7);
\draw (1.3,-3.8)--(1.8,-4.7);

\node at (0.25, -4.7) {$\vdots$};
\node at (1.55, -4.7) {$\vdots$};

\draw (0,-5.1)--(0.5,-6);
\draw (1.3,-5.1)--(1.8,-6);

\filldraw (1.55,-1.65) circle (0.25pt);
\filldraw (1.55,-1.75) circle (0.25pt);
\filldraw (1.55,-1.85) circle (0.25pt);
\filldraw (1.55,-1.95) circle (0.25pt);
\filldraw (1.55,-2.05) circle (0.25pt);
\filldraw (1.55,-2.15) circle (0.25pt);
\filldraw (1.55,-2.25) circle (0.25pt);
\filldraw (1.55,-2.35) circle (0.25pt);
\filldraw (1.55,-2.45) circle (0.25pt);
\filldraw (1.55,-2.55) circle (0.25pt);
\filldraw (1.55,-2.65) circle (0.25pt);
\filldraw (1.55,-2.75) circle (0.25pt);



\node at (0.25, -2.9) {$\vdots$};

\node[left] at (0,0) {$+1$};
\node[right] at (0.35, -0.25) {$-a_{11}$};
\node[right] at (1.65, -0.25) {$-a_{21}$};
\node[right] at (3.1, -0.35) {$-1$};

\node[right] at (0.35, -0.95) {$-a_{12}$};
\node[right] at (1.65, -0.95) {$-a_{22}$};

\node[right] at (0.35, -2.15) {$-a_{1n}'$};
\node[left] at (0.25, -1.65) {$C^1_{n}$};


\node[right] at (0.25, -5.45) {$-a_{1n_1}$};
\node[right] at (1.55, -5.45) {$-a_{2n_2}$};

\end{scope}
\begin{scope}[shift={(5.2,0)}]
\draw(0,0)--(3.5,0);
\draw (0.5,0.2)--(0,-0.7);
\draw (1.8,0.2)--(1.3,-0.7);
\draw (3.2,0.2)--(2.7,-2.5);

\draw[red,thick, dashdotted](3,-2.4)--(1.75,-1.44);
\draw[red,thick, dashdotted](1.35,-1.15)--(0.1,-0.2);
\node[red,below] at (3,-2.4){$e_2$};

\draw (0,-0.3)--(0.5,-1.2);
\draw (1.3,-0.3)--(1.8,-1.2);

\node at (0.25, -1.3) {$\vdots$};
\node at (1.55, -1.3) {$\vdots$};

\draw (0.5,-2)--(0,-2.9);
\draw (1.8,-2)--(1.3,-2.9);

\draw (0.,-2.5)--(0.5,-3.4);
\draw (1.3,-2.5)--(1.8,-3.4);

\draw[red,thick, dashdotted] (0.2,-2.1)--(1.8,-2.4);
\node[red,below] at (1.8,-2.4){$e$};

\node[left] at (0,0) {$+1$};
\node[right] at (0.35, -0.25) {$-a_{11}'$};
\node[right] at (1.65, -0.25) {$-a_{21}$};
\node[right] at (3.1, -0.35) {$-1$};

\node[right] at (1.65, -0.95) {$-a_{22}$};

\draw (0,-5.1)--(0.5,-6);
\draw (1.3,-5.1)--(1.8,-6);

\node[right] at (0.25, -5.45) {$-a_{1n_1}$};
\node[right] at (1.55, -5.45) {$-a_{2n_2}$};

\filldraw (0.25,-3.7) circle (0.25pt);
\filldraw (0.25,-3.8) circle (0.25pt);
\filldraw (0.25,-3.9) circle (0.25pt);
\filldraw (0.25,-4) circle (0.25pt);
\filldraw (0.25,-4.1) circle (0.25pt);
\filldraw (0.25,-4.2) circle (0.25pt);
\filldraw (0.25,-4.3) circle (0.25pt);
\filldraw (0.25,-4.4) circle (0.25pt);
\filldraw (0.25,-4.5) circle (0.25pt);
\filldraw (0.25,-4.6) circle (0.25pt);

\filldraw (1.55,-3.7) circle (0.25pt);
\filldraw (1.55,-3.8) circle (0.25pt);
\filldraw (1.55,-3.9) circle (0.25pt);
\filldraw (1.55,-4) circle (0.25pt);
\filldraw (1.55,-4.1) circle (0.25pt);
\filldraw (1.55,-4.2) circle (0.25pt);
\filldraw (1.55,-4.3) circle (0.25pt);
\filldraw (1.55,-4.4) circle (0.25pt);
\filldraw (1.55,-4.5) circle (0.25pt);
\filldraw (1.55,-4.6) circle (0.25pt);


\end{scope}
\end{tikzpicture}
\caption{Part of intermediate configuration $C'$}
\label{case2c'}
\end{center}
\end{figure}

By reindexing if needed, we may assume that the first and the second arm of configurations in Figure~\ref{startingposition} become that of $K$ respectively, and the proper transform of $e_2$ becomes an irreducible component in the third arm of $K$ after blowing-ups. For convenience, we omit all exceptional $(-1)$ strands that intersect only one irreducible component of the corresponding concave cap $K$ in figures.

Now, by blowing-ups at intersection points of $e_1$ consecutively, we get the first and the second arm so that the homological data for the irreducible components in these arms agrees with that of $C$ except for one irreducible component, say $C^1_n$, of the first arm of $K$ leaving the third single $(-1)$ arm unchanged. Note that there is only one exceptional strand $e$ connecting the first and the second arm as in Figure~\ref{case2c'} because we blow up at intersection points of $e_1$ to get the first and the second arm of $K$. Hence we can arrange a sequence of blowing-ups from a configuration in Figure~\ref{startingposition} to a curve configuration $C$ of type C so that we have an intermediate configuration $C'$ of strands as in Figure~\ref{case2c'}: The left-hand/right-hand figure is coming from $(a)'$/$(b)$ in Figure~\ref{startingposition} respectively.  For simplicity, we only explain a curve configuration coming from $(a)'$ in Figure~\ref{startingposition}. On contrary to the type B  case, we have a $(-a_{1n}')$ strand with $a_{1n}>a_{1n}'$ in $C'$ because we need to blow up at the intersection point of $e_2$ and $c$ in Figure~\ref{startingposition}, which becomes $(-a_{1n})$ strand in the curve configuration $C$ in Figure~\ref{case2c}. 
\begin{figure}[h]
\begin{center}
\begin{tikzpicture}[xscale=1.1,yscale=.8]
\begin{scope}
\draw(0,0)--(3,0);
\draw (0.5,0.2)--(0,-0.7);
\draw (1.5,0.2)--(1,-0.7);
\draw (2.5,0.2)--(2,-0.7);

\draw (0,-0.3)--(0.5,-1.2);
\draw (1,-0.3)--(1.5,-1.2);
\draw (2,-0.3)--(2.5,-1.2);

\node at (0.25, -1.2) {$\vdots$};
\node at (2.25, -1.2) {$\vdots$};

\draw (0.5,-1.6)--(0.,-2.5);
\draw (2.5,-1.6)--(2.,-2.5);

\draw (0.,-2.)--(0.5,-2.9);
\draw (2.,-2.)--(2.5,-2.9);

\draw (0.5,-3.3)--(0,-4.2);
\draw (1.5,-3.3)--(1,-4.2);

\draw[red,thick, dashdotted] (0.2,-3.4)--(1.5,-3.7);
\node[below,red] at (1.5, -3.7) {$e$};

\draw (0,-3.8)--(0.5,-4.7);
\draw (1,-3.8)--(1.5,-4.7);

\node at (0.25, -4.7) {$\vdots$};
\node at (1.25, -4.7) {$\vdots$};

\draw (0,-5.1)--(0.5,-6);
\draw (1,-5.1)--(1.5,-6);
\draw (2,-5.1)--(2.5,-6);

\filldraw (1.25,-1.65) circle (0.25pt);
\filldraw (1.25,-1.75) circle (0.25pt);
\filldraw (1.25,-1.85) circle (0.25pt);
\filldraw (1.25,-1.95) circle (0.25pt);
\filldraw (1.25,-2.05) circle (0.25pt);
\filldraw (1.25,-2.15) circle (0.25pt);
\filldraw (1.25,-2.25) circle (0.25pt);
\filldraw (1.25,-2.35) circle (0.25pt);
\filldraw (1.25,-2.45) circle (0.25pt);
\filldraw (1.25,-2.55) circle (0.25pt);
\filldraw (1.25,-2.65) circle (0.25pt);
\filldraw (1.25,-2.75) circle (0.25pt);

\filldraw (2.25,-3.6) circle (0.25pt);
\filldraw (2.25,-3.5) circle (0.25pt);

\filldraw (2.25,-3.7) circle (0.25pt);
\filldraw (2.25,-3.8) circle (0.25pt);
\filldraw (2.25,-3.9) circle (0.25pt);
\filldraw (2.25,-4) circle (0.25pt);
\filldraw (2.25,-4.1) circle (0.25pt);
\filldraw (2.25,-4.2) circle (0.25pt);
\filldraw (2.25,-4.3) circle (0.25pt);
\filldraw (2.25,-4.4) circle (0.25pt);
\filldraw (2.25,-4.5) circle (0.25pt);
\filldraw (2.25,-4.6) circle (0.25pt);

\node at (0.25, -2.9) {$\vdots$};

\node[left] at (0,0) {$+1$};
\node[right] at (0.25, -0.25) {$-a_{11}$};
\node[right] at (1.25, -0.25) {$-a_{21}$};
\node[right] at (2.25, -0.25) {$-a_{31}$};

\node[right] at (0.25, -0.9) {$-a_{12}$};
\node[right] at (1.25, -0.9) {$-a_{22}$};
\node[right] at (2.25, -0.9) {$-a_{32}$};

\draw[red,thick, dashdotted] (0.2,-1.8)--(1.05,-1.8);
\draw[red,thick, dashdotted] (1.45,-1.8)--(2.55,-1.8);
\node[left] at (0.25, -1.7) {$C^1_{n}$};

\node[right] at (0.25, -2.15) {$-a_{1n}$};

\node[right] at (0.25, -5.45) {$-a_{1n_1}$};
\node[right] at (1.25, -5.45) {$-a_{2n_2}$};
\node[right] at (2.25, -5.45) {$-a_{3n_3}$};

\end{scope}
\begin{scope}[shift={(5,0)}]
\draw(0,0)--(3,0);
\draw (0.5,0.2)--(0,-0.7);
\draw (1.5,0.2)--(1,-0.7);
\draw (2.5,0.2)--(2,-0.7);

\draw[red,thick, dashdotted] (0.1,-0.2)--(1.05,-1.1);
\draw[red,thick, dashdotted] (1.45,-1.5)--(2.5,-2.5);

\draw (0,-0.3)--(0.5,-1.2);
\draw (1,-0.3)--(1.5,-1.2);
\draw (2,-0.3)--(2.5,-1.2);

\node at (0.25, -1.3) {$\vdots$};
\node at (1.25, -1.3) {$\vdots$};
\node at (2.25, -1.3) {$\vdots$};

\draw (0.5,-2)--(0,-2.9);

\draw (1.5,-2)--(1.,-2.9);
\draw (2.5,-2)--(2.,-2.9);

\draw (0.,-2.5)--(0.5,-3.4);
\draw (1.,-2.5)--(1.5,-3.4);
\draw (2.,-2.5)--(2.5,-3.4);

\draw[red,thick, dashdotted] (0.2,-2.1)--(1.5,-2.4);
\node[below,red] at (1.5, -2.4) {$e$};

\node[left] at (0,0) {$+1$};
\node[right] at (0.25, -0.25) {$-a_{11}$};
\node[right] at (1.25, -0.25) {$-a_{21}$};
\node[right] at (2.25, -0.25) {$-a_{31}$};

\node[right] at (1.25, -0.9) {$-a_{22}$};
\node[right] at (2.25, -0.9) {$-a_{32}$};

\draw (0,-5.1)--(0.5,-6);
\draw (1,-5.1)--(1.5,-6);
\draw (2,-5.1)--(2.5,-6);

\node[right] at (0.25, -5.45) {$-a_{1n_1}$};
\node[right] at (1.25, -5.45) {$-a_{2n_2}$};
\node[right] at (2.25, -5.45) {$-a_{3n_3}$};

\filldraw (0.25,-3.7) circle (0.25pt);
\filldraw (0.25,-3.8) circle (0.25pt);
\filldraw (0.25,-3.9) circle (0.25pt);
\filldraw (0.25,-4) circle (0.25pt);
\filldraw (0.25,-4.1) circle (0.25pt);
\filldraw (0.25,-4.2) circle (0.25pt);
\filldraw (0.25,-4.3) circle (0.25pt);
\filldraw (0.25,-4.4) circle (0.25pt);
\filldraw (0.25,-4.5) circle (0.25pt);
\filldraw (0.25,-4.6) circle (0.25pt);

\filldraw (1.25,-3.7) circle (0.25pt);
\filldraw (1.25,-3.8) circle (0.25pt);
\filldraw (1.25,-3.9) circle (0.25pt);
\filldraw (1.25,-4) circle (0.25pt);
\filldraw (1.25,-4.1) circle (0.25pt);
\filldraw (1.25,-4.2) circle (0.25pt);
\filldraw (1.25,-4.3) circle (0.25pt);
\filldraw (1.25,-4.4) circle (0.25pt);
\filldraw (1.25,-4.5) circle (0.25pt);
\filldraw (1.25,-4.6) circle (0.25pt);

\filldraw (2.25,-3.7) circle (0.25pt);
\filldraw (2.25,-3.8) circle (0.25pt);
\filldraw (2.25,-3.9) circle (0.25pt);
\filldraw (2.25,-4) circle (0.25pt);
\filldraw (2.25,-4.1) circle (0.25pt);
\filldraw (2.25,-4.2) circle (0.25pt);
\filldraw (2.25,-4.3) circle (0.25pt);
\filldraw (2.25,-4.4) circle (0.25pt);
\filldraw (2.25,-4.5) circle (0.25pt);
\filldraw (2.25,-4.6) circle (0.25pt);


\end{scope}
\end{tikzpicture}
\caption{Part of curve configuration $C$ for $W$}
\label{case2c}
\end{center}
\end{figure}

Let $C_1$ be a curve configuration obtained from $C'$ by standard blowing-ups and $W_1$ be a minimal symplectic filling of $Y$ corresponding to $C_1$. 
Then, by Lemma~\ref{fundamentallem3}, there is a symplectic embedding $L$ in $W_1$ so that $W$ is obtained from $W_1$ by replacing $L$ with its minimal symplectic filling $W_L$, where $L$ is a plumbing graph in Figure~\ref{detail1}. Since a curve configuration $C_1$ for $W_1$ is of type B, there is a sequence of rational blowdowns from $\widetilde{M}$ to $W$ as desired.

\begin{figure}[htbp]
\begin{tikzpicture}[scale=0.9]
\begin{scope}[scale=0.6]
\node[bullet] at (0,1.5){};

\node[bullet] at (-2,0){};
\node[bullet] at (2,0){};

\node[bullet] at (-2,-1){};
\node[bullet] at (2,-1){};

\node[bullet] at (-2,-3){};
\node[bullet] at (2,-3){};

\draw (0,1.5)--(2,0)--(2,-1.5);
\draw (0,1.5)--(-2,0)--(-2,-1.5);

\draw[dotted](2,-1.5)--(2,-2.5);
\draw[dotted](-2,-1.5)--(-2,-2.5);

\draw(2,-2.5)--(2,-3);
\draw(-2,-2.5)--(-2,-3);
\draw (0,1.5) node[above] {$-b_{31}-1$};

\draw (-2,0) node[right] {$-2$};
\draw (-2,-1) node[right] {$-2$};
\draw (-2,-3) node[right] {$-2$};
\draw (2,0) node[right] {$-b_{32}$};
\draw (2,-1) node[right] {$-b_{33}$};
\draw (2,-3) node[right] {$-b_{3r_3}$};
	\draw [decorate,decoration={brace,mirror, amplitude=5pt},xshift=-10pt,yshift=0pt]
	(-2,0) -- (-2,-3) node [black,midway,xshift=-37pt] 
	{$a_{1n}-a_{1n}'-1$};
\end{scope}

\end{tikzpicture}
\caption{A plumbing graph of $L$}
\label{detail1}
\end{figure}

The following example illustrates this case.

\begin{example}
We consider a minimal symplectic filling $W_5$ of $Y$ in Example~\ref{example1}, represented by a curve configuration $C_5$ in Figure~\ref{finalposition}. The curve configuration $C_5$ is obtained from the right-hand configuration in Figure~\ref{startingposition}, and the proper transforms of $e_1$ and $e_2$ are irreducible components of the concave cap $K$. Thus, as in the proof, we can find intermediate configuration $C'$ between the right-hand configuration in Figure~\ref{startingposition} and $C_5$. Then it is easy to check that the homological data for standard blowing-ups $\widetilde{C'}$ of $C'$ and that of $C_1$ is equal (See Figure~\ref{middleposition} and Figure~\ref{finalposition}).
\begin{figure}[h]
\begin{tikzpicture}[scale=0.95]
\begin{scope}
\draw (0,0.2)--(2.8,0.2);
\draw (0.4,0.5)--(0.4,-2.4);

\draw (1.4,0.5)--(1.4,-0.6);
\draw (1.9,0.5)--(1.9,-1.3);
\draw (2.4,0.5)--(2.4,-2.4);

\draw (0.3,-0.4)--(1,-0.8);
\draw[red,thick, dashdotted] (0.8,-0.8)--(1.5,-0.4);

\draw[red,thick, dashdotted] (1.2,0)--(2.6,0);

\draw[red,thick, dashdotted] (0.2,-1.1)--(2.1,-1.1);
\node[left,red] at (0.2,-1.1){$e$};
\draw[red,thick, dashdotted] (0.2,-2.2)--(2.6,-2.2);

\node[above] at (1.9,0.5) {$-1$};
\node[above] at (2.4,0.5) {$-1$};
\node[left] at (0,0.2) {$+1$};
\node at (1.4,-2.75) {$C'$};

\draw[very thick, ->, >=stealth] (3.3,-0.85)--(4.3,-0.85);

\end{scope}
\begin{scope}[shift={(5,0)}]
\draw (0,0.2)--(2.8,0.2);
\draw (0.4,0.5)--(0.4,-2.4);

\draw (1.4,0.5)--(1.4,-0.6);
\node[right] at (1.35,-0.2) {$C^2_1$};

\draw (1.9,0.5)--(1.9,-1.3);
\node[right] at (1.85,-0.5) {$C^3_1$};

\draw (2.4,0.5)--(2.4,-2.4);
\node[right] at (2.45,-0.5) {$C^4_1$};

\draw (0.3,-0.4)--(1,-0.8);
\node[above] at (0.7,-0.6) {$C^1_2$};

\draw[red,thick, dashdotted] (0.8,-0.8)--(1.5,-0.4);

\draw[red,thick, dashdotted] (1.2,0.075)--(2.6,0.075);

\draw[red,thick, dashdotted] (0.2,-2.2)--(2.6,-2.2);

\node[above] at (2.4,0.5) {$-1$};
\node[left] at (0,0.2) {$+1$};
\node[left] at (0.4,-.75) {$C^1_1$};

\draw (2,-1.1)--(1.3,-1.6);
\node[below] at (2,-1.3) {$C^3_2$};

\draw[red,thick, dashdotted] (0.3,-1.8)--(1,-1.3);
\draw(0.8,-1.3)--(1.5,-1.6);
\node[below] at (1,-1.5) {$C^3_3$};

\draw[red,thick, dashdotted] (1.75,-1.45)--(1.55,-1.25);
\node at (1.4,-2.75) {$C_5$};

\end{scope}
\end{tikzpicture}
\caption{Curve configuration $C_5$ for symplectic filling $W_5$ of $Y$}
\label{finalposition}
\end{figure}
From the proof of Lemma~\ref{fundamentallem3}, we can explicitly check that there is a symplectic embedding of $L_1$:
\begin{tikzpicture}[scale=0.8]
\node[bullet] at (2,0){};
\node[bullet] at (3,0){};

\node[above] at (2,0){$-5$};
\node[above] at (3,0){$-2$};

\draw (2,0)--(3,0);
\end{tikzpicture} to $W_1$ in Example~\ref{example2}, and $W_5$ is obtained by rationally blowing down it:
Let $C_i^j$ be an $i^{\text{th}}$ component of the $j^{\text{th}}$ arm in $K$. Then, the homological data of $K$ for $W_1$ in $X=W_1\cup K \cong\mathbb{CP}^2 \sharp 10\overline{\mathbb{CP}^2}$ is given by
\begin{align*}
[C_0]&=l\\
[C_1^1]&= l-e_2-e_3-e_4-e_5, \phantom{0} [C_2^1]= e_2-e_6\\
[C_1^2]&= l-e_1-e_2-e_6\\
[C_1^3]&= l-e_1-e_3-e_7, \phantom{0} [C_2^3]= e_7-e_8, \phantom{0} [C_3^3]= e_8-e_9-e_{10}\\
[C_1^4]&= l-e_1-e_4,
\end{align*} where $C_0$ is the central $(+1)$ $2$-sphere of $K$, $l$ is the homology class representing the complex line in $\mathbb{CP}^2$, and $e_i$ is the homology class of each exceptional $2$-sphere. As in the proof of Lemma~\ref{fundamentallem3}, we can find a symplectic embedding of $L=$ \begin{tikzpicture}[scale=0.8]
\node[bullet] at (2,0){};
\node[bullet] at (3,0){};

\node[above] at (2,0){$-5$};
\node[above] at (3,0){$-2$};

\draw (2,0)--(3,0);
\end{tikzpicture} 
to $W_1\subset X$ whose homological data is given by $e_3-e_5-e_7-e_8-e_9$ and $e_9-e_{10}$ (refer to Figure~\ref{findingL1}). 
\begin{figure}[h]
\begin{tikzpicture}[scale=1.0]
\begin{scope}
\draw (0,0.2)--(2.8,0.2);
\draw (0.4,0.5)--(0.4,-2.4);

\draw (1.4,0.5)--(1.4,-0.6);
\draw (1.9,0.5)--(1.9,-1.3);
\draw (2.4,0.5)--(2.4,-2.4);

\draw (0.3,-0.4)--(1,-0.8);
\draw[red,thick, dashdotted] (0.8,-0.8)--(1.5,-0.4);

\draw[red,thick, dashdotted] (1.2,0)--(2.6,0);

\draw[red,thick, dashdotted] (0.2,-1.1)--(2.1,-1.1);
\node[left,red] at (0.2,-1.1){$e$};
\filldraw[red] (0.4,-1.1) circle (1.5pt);
\filldraw[red] (1.9,-1.1) circle (1.5pt);

\draw[red,thick, dashdotted] (0.2,-2.2)--(2.6,-2.2);

\node[above] at (1.9,0.5) {$-1$};
\node[above] at (2.4,0.5) {$-1$};
\node[left] at (0,0.2) {$+1$};

\draw[very thick, ->, >=stealth] (2.95,-0.85)--(3.45,-0.85);

\end{scope}
\begin{scope}[shift={(3.75,0)}]
\draw (0,0.2)--(2.8,0.2);
\draw (0.4,0.5)--(0.4,-2.4);

\draw (1.4,0.5)--(1.4,-0.6);
\draw (1.9,0.5)--(1.9,-1.3);
\draw (2.4,0.5)--(2.4,-2.4);

\draw (0.3,-0.4)--(1,-0.8);

\draw[red,thick, dashdotted] (0.8,-0.8)--(1.5,-0.4);

\draw[red,thick, dashdotted] (1.2,0)--(2.6,0);

\draw (1.8,-.9)--(2.1,-1.4);
\draw(2.1,-1.3)--(1.8,-1.8);
\draw[red,thick, dashdotted] (2.,-1.8)--(1.6,-1.3);
\draw[red,thick, dashdotted] (0.7,-1.8)--(0.3,-1.3);

\draw[red,thick, dashdotted] (0.2,-2.2)--(2.6,-2.2);
\draw(0.5,-1.7)--(1.8,-1.4);
\node at (1.15,-1.8){$-5$};
\node at (2.15,-1.05){$-2$};
\node at (2.15,-1.6){$-2$};
\filldraw[red] (1.88,-1.65) circle (1.1pt);

\node[above] at (2.4,0.5) {$-1$};
\node[left] at (0,0.2) {$+1$};

\draw[very thick, ->, >=stealth] (2.95,-0.85)--(3.45,-0.85);

\end{scope}
\begin{scope}[shift={(7.5,0)}]
\draw (0,0.2)--(2.8,0.2);
\draw (0.4,0.5)--(0.4,-2.4);

\draw (1.4,0.5)--(1.4,-0.6);
\draw (1.9,0.5)--(1.9,-1.3);
\draw (2.4,0.5)--(2.4,-2.4);

\draw (0.3,-0.4)--(1,-0.8);

\draw[red,thick, dashdotted] (0.8,-0.8)--(1.5,-0.4);

\draw[red,thick, dashdotted] (1.2,0)--(2.6,0);

\draw (1.8,-.9)--(2.1,-1.4);
\draw(2.1,-1.3)--(1.8,-1.8);
\draw[red,thick, dashdotted] (2.,-1.8)--(1.6,-1.3);
\draw[red,thick, dashdotted] (0.7,-1.8)--(0.3,-1.3);

\draw[red,thick, dashdotted] (0.2,-2.2)--(2.6,-2.2);
\draw(0.5,-1.7)--(1.25,-1.3);
\draw(1.,-1.3)--(1.9,-1.5);
\node at (1,-1.7){$-5$};
\node at (1.5,-1.7){$-2$};
\node at (2.15,-1.05){$-2$};
\node at (2.15,-1.6){$-3$};

\node[above] at (2.4,0.5) {$-1$};
\node[left] at (0,0.2) {$+1$};
\end{scope}
\end{tikzpicture}
\caption{Embedding of $L_1$ in $W_1$}
\label{findingL1}
\end{figure}
There are two minimal symplectic fillings of $L$ whose corresponding curve configurations are as in Figure~\ref{lens52}. Note that the first figure represents a linear plumbing while the second figure represents a rational homology $4$-ball.
\begin{figure}[h]
\begin{tikzpicture}[scale=0.6]
\begin{scope}
\draw (0.5,0.2)--(0,-1.);
\draw (0,-0.6)--(0.5,-1.8);
\draw (0.5,-1.4)--(0,-2.6);
\draw (0,-2.2)--(0.5,-3.4);
\draw (0.5,-2.8)--(0,-4);

\draw[red,thick, dashdotted] (0.0,-1.2)--(0.5,-1.2);
\draw[red,thick, dashdotted] (0.03,-3.5)--(0.43,-3.5);
\draw[red,thick, dashdotted] (0.0,-3.7)--(0.4,-3.7);

\node[left] at (-0.3,-.2) {$+1$};
\node[left] at (-0.3,-1.1) {$-1$};
\node[left] at (-0.3,-2) {$-2$};
\node[left] at (-0.3,-2.9) {$-2$};
\node[left] at (-0.3,-3.8) {$-3$};
\end{scope}
\begin{scope}[shift={(4,0)}]
\draw (0.5,0.2)--(0,-1.);
\draw (0,-0.6)--(0.5,-1.8);
\draw (0.5,-1.4)--(0,-2.6);
\draw (0,-2.2)--(0.5,-3.4);
\draw (0.5,-2.8)--(0,-4);
\draw[red,thick, dashdotted] (0.0,-2.7)--(0.5,-2.7);


\node[left] at (-0.3,-.2) {$+1$};
\node[left] at (-0.3,-1.1) {$-1$};
\node[left] at (-0.3,-2) {$-2$};
\node[left] at (-0.3,-2.9) {$-2$};
\node[left] at (-0.3,-3.8) {$-3$};
\end{scope}
\end{tikzpicture}
\caption{Two curve configurations for $Y_L$}
\label{lens52}
\end{figure}Hence, if we rationally blow down $L$ from $X_L=L \cup K_L\cong \mathbb{CP}^2\sharp 6\overline{\mathbb{CP}^2}$, then we get a new rational symplectic $4$-manifold $X_L'\cong \mathbb{CP}^2\sharp 4\overline{\mathbb{CP}^2}$ and the homological data of $K_L$ changes as follows: \begin{align*}
l&\rightarrow l\\
l-e_1-e_2&\rightarrow l-E_1-E_2\\
e_2-e_3&\rightarrow E_2-E_3\\
e_3-e_4&\rightarrow E_3-E_4\\
e_4-e_5-e_{6}&\rightarrow E_1-E_2-E_3
\end{align*}
Here $e_i$ and $E_i$ denote the homology classes of exceptional spheres in $X_L$ and $X_L'$. Note that homological data of $L$ in $X_L$ is given by $e_1-e_2-e_3-e_4-e_5$ and $e_5-e_6$.
Therefore, if we see $X$ as $X_L\sharp 4 \overline{\mathbb{CP}^2}$, we get $X'\cong \mathbb{CP}^2 \sharp 8\overline{\mathbb{CP}^2}$ by rationally blowing down $L$ from $X$ and the homological data of $K_L$ is changed by \begin{align*}
l&\rightarrow l\\
l-e_3-e_5&\rightarrow l-E_1-E_2\\
e_5-e_7&\rightarrow E_2-E_3\\
e_7-e_8&\rightarrow E_3-E_4\\
e_8-e_9-e_{10}&\rightarrow E_1-E_2-E_3,
\end{align*}
where $e_1,e_2, e_4, e_6$ and $E_1, E_2, E_3, E_4$ represent the standard exceptional $2$-spheres in $X'\cong\mathbb{CP}^2\sharp 8\overline{\mathbb{CP}^2}$.
Therefore, the new homological data for concave cap $K$, which give the right-hand curve configuration in Figure~\ref{finalposition}, are as follows:
\begin{align*}
[C_0]&=l\\
[C_1^1]&= l-e_2-e_4-E_1-E_2, \phantom{0} [C_2^1]= e_2-e_6\\
[C_1^2]&= l-e_1-e_2-e_6\\
[C_1^3]&= l-e_1-E_1-E_3, \phantom{0} [C_2^3]= E_3-E_4, \phantom{0} [C_3^3]= E_1-E_2-E_3\\
[C_1^4]&= l-e_1-e_4
\end{align*}
\label{example3}
\end{example}
\begin{remark}
In fact, we investigated all possible curve configurations for a small Seifert $3$-manifold $Y$ with $b\geq 5$ in the proof of main theorem. As a result, we can find all minimal symplectic fillings of $Y$ via corresponding curve configurations. 
For example, a complete list of minimal symplectic fillings of $Y$ in Example~\ref{example1} are given by Example~\ref{example1}, Example~\ref{example2} and Example~\ref{example3}.
\end{remark}\smallskip

\subsection{Proof for Type D}

We start to prove this case for a curve configuration coming from $C_{0,0,0}$. Note that $C_{0,0,0}$ itself is a curve configuration containing $K_{0,0,0}$ corresponding to rational homology ball filling of $\Gamma_{0,0,0}$ in Figure~\ref{Gammapqr}. By repeatedly blowing-up at intersection points between exceptional strands and the first component of each arm, we can get a curve configuration $C_{p,q,r}$ containing $K_{p,q,r}$ corresponding to a rational homology ball filling of $\Gamma_{p,q,r}$ as in Figure~\ref{b4cases}. For notational convenience, we denote three exceptional strands in each $C_{p,q,r}$ by the same $e_i$ with $i\in \mathbb{Z}_3$ so that $e_i$ intersects the last component of $i^{\text{th}}$ arm and the first component of $(i+1)^{\text{th}}$ arm of $K_{p,q,r}$. Let $C_{-1,-1,-1}$ be just right-hand figure of Figure~\ref{startingposition} and $C_{p,q,-1}$ be a configuration of strands obtained from $C_{p,q,0}$ by blowing down $e_2$ in Figure~\ref{b4cases}. Then $C_{p,q,-1}$ contains $K_{p,q,-1}$ which is the proper transform of $K_{p,q,0}$ under the blowing down.

\begin{figure}[h]

\begin{tikzpicture}[scale=.9]
\begin{scope}
\draw (-0.5,0)--(3.5,0);
\node[left] at (-0.5,0){$+1$};
\draw (0,0.5)--(0,-3);
\node[above] at (0,0.5){$-2$};
\draw (1.5,0.5)--(1.5,-2);
\node[above] at (1.5,0.5){$-2$};
\draw (3,0.5)--(3,-3);
\node[above] at (3,.5){$-2$};

\draw (-0.2,-1.2)--(0.95,-2.35);
\node[below] at (0.95,-2.35){$-2$};

\draw[red,thick, dashdotted] (1.7,-1.2)--(0.55,-2.35);
\node[red, above] at (1.025,-1.775){$e_1$};

\draw (1.3,-0.3)--(2.45,-1.45);
\node[below] at (2.45,-1.45){$-2$};
\draw[red,thick, dashdotted] (3.2,-0.3)--(2.05,-1.45);
\node[red, above] at (2.525,-0.875){$e_2$};

\draw[red,thick, dashdotted] (-0.2,-2.1)--(1.7,-4);
\node[red, below] at (0.64,-3.05){$e_3$};
\draw (3.2,-2)--(1.3,-4);
\node[below] at (2.35,-3){$-2$};

\end{scope}

\begin{scope}[shift={(6,0)}]

\draw (-0.5,0)--(3.5,0);
\node[left] at (-0.5,0){$+1$};
\draw (0,0.5)--(0,-3);
\node[above] at (0,0.5){$-(r+2)$};
\draw (1.5,0.5)--(1.5,-2);
\node[above] at (1.5,0.5){$-(p+2)$};
\draw (3,0.5)--(3,-3);
\node[above] at (3,.5){$-(q+2)$};

\draw (-0.1,-1.4)--(0.3,-1.4);
\draw (0.2,-1.3)--(0.2,-1.7);
\node at (0.45,-1.65) {$\ddots$};

\draw (0.45,-1.95)--(0.85,-1.95);
\draw (0.75,-1.85)--(0.75,-2.25);

	\draw [decorate,decoration={brace,amplitude=5pt},yshift=10pt,xshift=5pt]
	(0,-1.5) -- (0.75,-2.25) node [black,midway,xshift=13pt, yshift=10pt] 
	{$p+1$};

\draw[red,thick, dashdotted] (1.7,-1.2)--(0.55,-2.35);
\node[red, below right] at (1.025,-1.775){$e_1$};


\draw (1.4,-0.5)--(1.8,-0.5);
\draw (1.7,-0.4)--(1.7,-0.8);
\node at (1.95,-0.75) {$\ddots$};

\draw (1.95,-1.05)--(2.35,-1.05);
\draw (2.25,-0.95)--(2.25,-1.35);

	\draw [decorate,decoration={brace,amplitude=5pt},yshift=10pt,xshift=5pt]
	(1.5,-0.6) -- (2.25,-1.35) node [black,midway,xshift=13pt, yshift=9pt] 
	{$q+1$};

\draw[red,thick, dashdotted] (3.2,-0.3)--(2.05,-1.45);
\node[red, below right] at (2.525,-0.875){$e_2$};

\draw[red,thick, dashdotted] (-0.2,-2.1)--(1.7,-4);
\draw(1.55,-3.35)--(1.95,-3.35);
\draw(1.65,-3.65)--(1.65,-3.25);
\draw(1.35,-3.55)--(1.75,-3.55);
\draw(1.45,-3.45)--(1.45,-3.85);

\node[red, below] at (0.64,-3.05){$e_3$};
\draw(3.1,-2.2)--(2.7,-2.2);
\draw(2.8,-2.1)--(2.8,-2.5);
\draw(2.9,-2.4)--(2.5,-2.4);
\draw(2.6,-2.3)--(2.6,-2.7);

	\draw [decorate,decoration={brace,mirror,amplitude=5pt},yshift=-5pt,xshift=5pt]
	(1.65,-3.65) -- (2.9,-2.4) node [black,midway,xshift=10pt, yshift=-10pt] 
	{$r+1$};

\filldraw (2.25,-2.95) circle(0.5pt);
\filldraw (2.4,-2.8) circle(0.5pt);
\filldraw (2.1,-3.1) circle(0.5pt);

\end{scope}
\end{tikzpicture}
\caption{Curve configurations $C_{0,0,0}$ and $C_{p,q,r}$}
\label{b4cases}

\end{figure}

\begin{prop}
For a curve configuration $C$ coming from $C_{0,0,0}$, there is a curve configuration $C_{a,b,c}$ containing $K_{a,b,c}$ with $a,b,c \geq -1$ such that
\begin{enumerate}[(i)]
\item there is a sequence of blowing-ups from $C_{a,b,c}$ to $C$,
\item there is either no blowing-up at $e_i$ or blowing-ups at both intersection points on $e_i$ during the sequence of blowing-ups,
\item there is no blowing-up at intersection points of $K_{a,b,c}$.
\end{enumerate}
\label{propb=4}
\end{prop}

\begin{proof}
Since there are no strands with degree $\leq -2$ in $C$ except for irreducible components of $K$, each irreducible component of $K_{0,0,0}$ in $C_{0,0,0}$ should become an irreducible component of $K$ under blowing-ups from $C_{0,0,0}$ to $C$.
Hence, in order to get $C$ from $C_{0,0,0}$ by blowing-ups $e_i$, we should blow up at either two intersection points of $e_i$ with arms or an intersection point of $e_i$ with $(i+1)^{\text{th}}$-arm only. Note that we get $C_{p,q,r}$ containing $K_{p,q,r}$ by blowing-ups the latter case repeatedly.
Hence, by rearranging the order of blowing-ups from $C_{0,0,0}$ to a curve configuration $C$, we may assume that $C$ is obtained from $C_{p,q,r}$ with $p,q,r\geq 0$ and there are no more blowing-ups at an intersection point of $e_i$ with $(i+1)^{\text{th}}$-arm only. Since the configuration $C_{p,q,r}$ clearly satisfies Conditions (i) and (ii), we are done if there is no blowing-up at intersection points of $K_{p,q,r}$ in $C_{p,q,r}$.

If there are blowing-ups at intersection points of $K_{p,q,r}$ in $C_{p,q,r}$ to $C$, then we will find another $C_{a,b,c}$ with $a \geq p$, $b\geq q$ and $c\geq r$ satisfying Conditions (i)-(iii) as follows.
Let $x_i$ be the first intersection point in the $i^{\text{th}}$ arm of $K_{p,q,r}$ among the intersection points to be blown up and $C'$ be a configuration of strands obtained by blowing up at $x_i$ $(1\leq i \leq 3)$.
For notational convenience, we denote exceptional strands in $C_{p,q,r}$ and the proper transform of $e_i$ in $C'$ by the same $e_i$.
There is a unique $(-1)$ exceptional strand in each $i^{\text{th}}$ arm of $K'$ in $C'$, which is the ${n_i}^{\text{th}}$-component of $i^{\text{th}}$-arm with $n_i \geq 2$, where $K'$ is the total transform of $K_{p,q,r}$.
Then we claim that there is a sequence of blowing-ups from $C_{n_1-3,n_2-3,n_3-3}$ to $C'$: We blow up two intersection points of $e_i$ simultaneously, and then we blow up at intersection point between exceptional $(-1)$ strand and the first component of $(i+1)^{\text{th}}$ arm consecutively to get $C'$  (For example, see Figure~\ref{detail4.4} for the first arm).  
\begin{figure}[h]

\begin{tikzpicture}[scale=0.7]
\draw[very thick, ->, >=stealth] (3.5,-1.75)--(4.5,-1.75);
\draw[very thick, ->, >=stealth] (9.25,-1.75)--(10.75,-1.75);

\begin{scope}[shift={(0,0)}]
\draw (3.,0.1)--(2.5,-0.9);
\draw (1.5,0.1)--(1,-0.9);
\draw (1,-0.3)--(1.5,-1.2);

\draw[red,thick, dashdotted] (3.,-0.2)--(1,-2.2);

\node at (1.25, -1.25) {$\vdots$};

\draw (1,-1.6)--(1.5,-2.6);

\draw [decorate,decoration={brace,mirror, amplitude=5pt},xshift=-10pt,yshift=0pt]
	(1,0) -- (1,-2.4) node [black,midway,xshift=-25pt] 
	{$(n_1-1)$};

\node[] at (3, 0.5) {$-(n_{1}-1)$};
\node[red] at (2.35, -1.6) {$e_1$};
\filldraw[red] (1.2,1.2-3.2) circle (1.75pt);
\filldraw[red] (2.7,2.7-3.2) circle (1.75pt);


\end{scope}

\begin{scope}[shift={(5.5,0.5)}]
\draw (3,0.)--(2.8,-1);
\draw (1.5,0.1)--(1,-0.9);
\draw (1,-0.3)--(1.5,-1.2);

\draw[red,thick, dashdotted] (1.25,-3.8)--(3.2,.1);
\node at (1.25, -1.25) {$\vdots$};

\draw (1,-1.6)--(1.5,-2.6);
\draw (1.5,-2.2)--(1,-3.2);
\draw (1,-2.8)--(1.5,-3.8);


\draw [decorate,decoration={brace,mirror, amplitude=5pt},xshift=-10pt,yshift=0pt]
	(1,0) -- (1,-2.1) node [black,midway,xshift=-25pt] 
	{$(n_1-1)$};

\node[] at (3, 0.5) {$-n_{1}$};
\node[] at (0.5, -2.6) {$-1$};
\node[] at (.5, -3.5) {$-3$};
\filldraw[red] (2.9,-.5) circle (1.75pt);


\end{scope}

\begin{scope}[shift={(12,1.5)}]
\draw (3,0.)--(2.5,-1);
\draw (1.5,0.1)--(1,-0.9);
\draw (1,-0.3)--(1.5,-1.2);

\node at (1.25, -1.25) {$\vdots$};
\node at (1.25, -4.35) {$\vdots$};

\draw (1,-1.6)--(1.5,-2.6);
\draw (1.5,-2.2)--(1,-3.2);
\draw (1,-2.8)--(1.5,-3.8);
\draw (1.5,-3.2)--(1.,-4.2);
\draw (1.5,-4.6)--(1,-5.6);

\draw[red,thick, dashdotted] (1,-5.4) to[in=260] (2.5,-3.5) to[out=80,in=265](2.84,0);

\draw [decorate,decoration={brace,mirror, amplitude=5pt},xshift=-10pt,yshift=0pt]
	(1,0) -- (1,-2.1) node [black,midway,xshift=-25pt] 
	{$(n_1-1)$};
\draw [decorate,decoration={brace,mirror, amplitude=5pt},xshift=-10pt,yshift=0pt]
	(1,-4) -- (1,-5.4) node [black,midway,xshift=-35pt] 
	{$(p-n_1)$};

\node[] at (3, 0.5) {$-p$};
\node[] at (0.5, -2.6) {$-1$};
\node[] at (.5, -3.5) {$-3$};


\end{scope}

\end{tikzpicture}

\caption{Part of blowing-ups from $C_{n_1-3,n_2-3,n_3-3}$ to $C'$}
\label{detail4.4} 
\end{figure}
We see from the construction that a configuration $C_{n_1-3,n_2-3,n_3-3}$ satisfies Conditions (i) and (ii). Moreover, since $x_i$ is uppermost point among the intersection points to be blown up, there is no blowing-up at intersection points of $K_{n_1-3,n_2-3,n_3-3}$ during the blowing-ups from $C_{n_1-3,n_2-3,n_3-3}$ to $C$. Therefore $C_{n_1-3,n_2-3,n_3-3}$ is a desired curve configuration $C_{a,b,c}$.
\end{proof}
Since $K_{a,b,c} \leq K$ (guaranteed by Condition (ii) in Proposition~\ref{propb=4}), there is a curve configuration $C_1$ of $Y$ obtained from $C_{a,b,c}$ by standard blowing-ups. If one of $a,b,c$ is $-1$, then the curve configuration $C_1$ is of type B or type C, so that there is a sequence of rational blowdowns from $\widetilde{M}$ to the minimal symplectic filling $W_1$ corresponding to $C_1$. If all $a,b,c\geq 0$, then $W_1$ is obtained from $\widetilde{M}$ by replacing $\Gamma_{a,b,c}$ with its rational homology ball filling by Lemma~\ref{fundamentallem2}.
On the other hand, Conditions (ii) and (iii) in Proposition~\ref{propb=4} guarantee that there is a sequence of rational blowdowns from $W_1$ to the minimal symplectic filling $W$ corresponding to $C$ by using Lemma~\ref{fundamentallem} or Lemma~\ref{fundamentallem3} repeatedly.\\

We end this section by giving an example of minimal symplectic fillings involving $3$-legged rational blowdown surgery. 

\begin{example}
Let $Y$ be a small Seifert $3$-manifold whose minimal resolution graph $\Gamma$ and concave cap $K$ are given by Figure~\ref{b4example}. 
We consider two minimal symplectic fillings $W_1, W_2$ of $Y$ whose curve configurations are given by Figure~\ref{b4example1} and Figure~\ref{b4example2}. 
Note that the curve configuration in Figure~\ref{b4example1} is obtained from $C_{0,0,0}$ by standard blowing-ups. Thus, as in the proof, $W_1$ is obtained from the minimal resolution by rationally blowing down $\Gamma_{0,0,0}$. 
Let us denote $v_0$ by a central vertex and $v_i^j$ by $i^{\text{th}}$-vertex of the $j^{\text{th}}$-arm in $\Gamma$.
Then, $v_0, v_1^1, v_1^2$ and $v_1^3+v_2^3$ give a symplectic embedding of $\Gamma_{0,0,0}$ to the minimal resolution. A computation similar to that of Example~\ref{example3} shows that there is a symplectic embedding $L$ of \begin{tikzpicture}[scale=0.7]
\node[bullet] at (2,0){};
\node[bullet] at (3,0){};

\node[above] at (2,0){$-5$};
\node[above] at (3,0){$-2$};

\draw (2,0)--(3,0);
\end{tikzpicture} to $W_1$ and $W_2$ is obtained from $W_1$ by rationally blowing down $L$.

\begin{figure}[h]
\begin{tikzpicture}[xscale=0.75, yscale=0.55]
\begin{scope}

\node[bullet] at (3,0){};
\node[bullet] at (4,0){};
\node[bullet] at (4,-1){};
\node[bullet] at (4,-2){};
\node[bullet] at (4,-3){};

\node[bullet] at (5,0){};
\node[bullet] at (6,0){};

\node[above] at (3,0){$-3$};
\node[above] at (4,0){$-4$};
\node[above] at (5,0){$-2$};
\node[above] at (6,0){$-3$};

\node[left] at (4,-1){$-3$};
\node[left] at (4,-2){$-4$};
\node[left] at (4,-3){$-2$};

\draw (3,0)--(6,0);
\draw (4,0)--(4,-3);
\end{scope}
\begin{scope}[shift={(8,0)}]
\draw(0,0)--(3,0);
\draw (0.5,0.2)--(0,-1.);
\draw (1.5,0.2)--(1,-1.);
\draw (2.5,0.2)--(2,-1.);

\draw (0,-0.6)--(0.5,-1.8);
\draw (1,-0.6)--(1.5,-1.8);
\draw (2,-0.6)--(2.5,-1.8);

\draw (1.5, -1.4)--(1,-2.6);

\draw (1,-2.2)--(1.5,-3.4);

\node[left] at (0,0) {$+1$};
\node[right] at (0.25, -0.25) {$-2$};
\node[right] at (1.25, -0.25) {$-2$};
\node[right] at (2.25, -0.25) {$-3$};

\node[right] at (0.25, -1) {$-2$};
\node[right] at (1.25, -1) {$-3$};
\node[right] at (1.25, -2.25) {$-2$};
\node[right] at (1.25, -3) {$-3$};

\node[right] at (2.25, -1) {$-2$};

\end{scope}
\end{tikzpicture}

\caption{Plumbing graph $\Gamma$ and its concave cap $K$}
\label{b4example}
\end{figure}

\begin{figure}[h]
\begin{tikzpicture}[scale=0.65]
\begin{scope}
\draw (-0.5,0)--(3.5,0);
\node[left] at (-0.5,0){$+1$};
\draw (0,0.5)--(0,-3);
\node[above] at (0,0.5){$-2$};
\draw (1.5,0.5)--(1.5,-2);
\node[above] at (1.5,0.5){$-2$};
\draw (3,0.5)--(3,-3);
\node[above] at (3,.5){$-2$};

\draw (-0.2,-1.2)--(0.95,-2.35);
\node[below] at (0.95,-2.35){$-2$};

\draw[red,thick, dashdotted] (1.7,-1.2)--(0.55,-2.35);
\node[red, above] at (1.025,-1.775){$e_1$};

\draw (1.3,-0.3)--(2.45,-1.45);
\node[below] at (2.45,-1.45){$-2$};
\draw[red,thick, dashdotted] (3.2,-0.3)--(2.05,-1.45);
\node[red, above] at (2.525,-0.875){$e_2$};

\draw[red,thick, dashdotted] (-0.2,-2.1)--(1.7,-4);
\node[red, below] at (0.64,-3.05){$e_3$};
\draw (3.2,-2)--(1.3,-4);
\node[below] at (2.35,-3){$-2$};
\draw[very thick, ->, >=stealth] (4,-1.25)--(5,-1.25);

\end{scope}

\begin{scope}[shift={(6,0)}]

\draw (-0.5,0)--(3.5,0);
\node[left] at (-0.5,0){$+1$};
\draw (0,0.5)--(0,-3);
\draw (1.5,0.5)--(1.5,-2);
\draw (3,0.5)--(3,-3);
\draw[red,thick, dashdotted] (2.75,-1.1)--(3.25,-1.1);
\draw (-0.2,-1.2)--(0.95,-2.35);

\draw[red,thick, dashdotted] (1.7,-1.2)--(0.55,-2.35);
\draw (1.3,-0.3)--(2.45,-1.45)--(2.55,-1.55);
\draw (2.55,-1.35)--(2,-1.9);
\draw (2,-1.7)--(2.65,-2.35);
\draw[red,thick, dashdotted] (2.5,-1.8)--(2.1,-2.2);
\draw[red,thick, dashdotted] (2.6,-1.9)--(2.2,-2.3);

\draw[red,thick, dashdotted] (3.2,-0.3)--(2.05,-1.45);

\draw[red,thick, dashdotted] (-0.2,-2.1)--(1.7,-4);
\draw (3.2,-2)--(1.3,-4);

\end{scope}
\end{tikzpicture}
\caption{Curve configuration for $W_1$}
\label{b4example1}
\end{figure}
\begin{figure}[h]
\begin{tikzpicture}[scale=0.65]
\begin{scope}
\draw (-0.5,0)--(3.5,0);
\node[left] at (-0.5,0){$+1$};
\draw (0,0.5)--(0,-3);
\node[above] at (0,0.5){$-2$};
\draw (1.5,0.5)--(1.5,-2);
\node[above] at (1.5,0.5){$-2$};
\draw (3,0.5)--(3,-3);
\node[above] at (3,.5){$-2$};

\draw (-0.2,-1.2)--(0.95,-2.35);
\node[below] at (0.95,-2.35){$-2$};

\draw[red,thick, dashdotted] (1.7,-1.2)--(0.55,-2.35);
\node[red, above] at (1.025,-1.775){$e_1$};

\draw (1.3,-0.3)--(2.45,-1.45);
\node[below] at (2.45,-1.45){$-2$};
\draw[red,thick, dashdotted] (3.2,-0.3)--(2.05,-1.45);
\node[red, above] at (2.525,-0.875){$e_2$};
\filldraw (3,-0.5) circle (1pt);
\filldraw (2.25,-1.25) circle (1pt);

\draw[red,thick, dashdotted] (-0.2,-2.1)--(1.7,-4);
\node[red, below] at (0.64,-3.05){$e_3$};
\draw (3.2,-2)--(1.3,-4);
\node[below] at (2.35,-3){$-2$};
\draw[very thick, ->, >=stealth] (4,-1.25)--(5,-1.25);

\end{scope}

\begin{scope}[shift={(6,0)}]

\draw (-0.5,0)--(3.5,0);
\node[left] at (-0.5,0){$+1$};
\draw (0,0.5)--(0,-3);
\draw (1.5,0.5)--(1.5,-2);
\draw (3,0.5)--(3,-3);
\draw (-0.2,-1.2)--(0.95,-2.35);

\draw[red,thick, dashdotted] (1.7,-1.2)--(0.55,-2.35);

\draw (1.3,-0.3)--(1.9,-0.9);
\draw[red,thick, dashdotted] (1.9, -0.5)--(2.1,-0.7);
\draw (1.7,-0.9)--(2.3,-0.3);
\draw (2.1,-0.3)--(2.7,-0.9);
\draw[red,thick, dashdotted] (2.5,-0.9)--(3.1,-0.3);

\draw[red,thick, dashdotted] (-0.2,-2.1)--(1.7,-4);
\draw (3.2,-2)--(1.3,-4);

\end{scope}
\end{tikzpicture}
\caption{Curve configuration for $W_2$}
\label{b4example2}
\end{figure}
\end{example}

\medskip


\providecommand{\bysame}{\leavevmode\hbox to3em{\hrulefill}\thinspace}

\end{document}